\documentclass[11pt]{amsart}
\usepackage{geometry}                
  \usepackage[utf8]{inputenc} 
\usepackage{graphicx}
\usepackage{amssymb}
\usepackage{epstopdf}
\newtheorem{thm}{Theorem}[section]
\newtheorem{lem}[thm]{Lemma}

\newtheorem{remark}[thm]{Remark}

\title[Reduced models for spectral inversion]{Reduced order models for spectral domain inversion: Embedding into the continuous problem and generation of internal data.}
\author{L. Borcea, V. Druskin, A.V. Mamonov, S. Moskow, M. Zaslavsky}
\begin{document}
\begin{abstract} We generate data-driven reduced order models (ROMs) for inversion of the one and two dimensional Schr\"odinger equation in the spectral domain given boundary data at a few frequencies. The ROM is the Galerkin projection of the Schr\"odinger operator onto the space spanned by solutions at these sample frequencies. The ROM matrix is in general full, and not good for extracting the potential. However, using an orthogonal change of basis via Lanczos iteration, we can transform the ROM to a block triadiagonal form from which it is easier to extract $q$. In one dimension, the tridiagonal matrix corresponds to a three-point staggered finite-difference system for the Schr\"odinger operator discretized on a so-called  spectrally matched grid which is almost independent of the medium. In higher dimensions, the orthogonalized basis functions play the role of the grid steps. The orthogonalized basis functions are localized and also depend only very weakly on the medium, and thus by embedding into the continuous problem, the reduced order model yields highly accurate internal solutions. That is to say, we can obtain, just from boundary data, very good approximations of the solution of the Schr\"odinger equation in the whole domain for a spectral interval that includes the sample frequencies. We present inversion experiments based on the internal solutions in one and two dimensions. \end{abstract}
\maketitle

\section{Introduction}
In this work we consider the following problem on a bounded domain $\Omega$ in $d$ dimensions  
\begin{eqnarray} \label{2dprob1} -\Delta u(x; \lambda )+q(x) u(x;\lambda ) +\lambda u(x; \lambda)  & = & 0  \ \ \mbox{for}\ \ \  x\in\Omega\subset \mathbb{R}^d \\   {\partial u\over {\partial \nu  }} (x;\lambda) & = & g\ \ \mbox{for} \ \  x\in \partial \Omega\nonumber 
\end{eqnarray}
for $\nu$ the outward unit normal to the boundary $\partial\Omega$ and $q\in L^\infty(\Omega)$. That is, $u(x; \lambda) $ is the solution to the Schr\"odinger equation with Neumann data corresponding to spectral parameter $\lambda$, which here we assume to be positive real. 
We apply $K$  Neumann boundary functions $\{ g_1, \ldots,g_K \} $, with each $g_r\in H^{-1/2}(\partial \Omega)$ and read the corresponding Dirichlet data $u(x; \lambda)$ on the boundary $\partial\Omega$ for multiple values of $\lambda$, and are interested in determining $q(x)$ from this data. 
The reconstruction of $q(x)$ from  boundary measurements is known to be an ill-posed problem and, depending on the physical setup can result in reconstructions with very poor resolution \cite{AS-IP-10}.  The map from $q(x)$ to the data is highly nonlinear, and due to ill-posedness is difficult to invert unless regularization is added.  As examples of literature on this subject, see \cite{Bo},\cite{ArSc},\cite{Uh}. 

A relatively new class of approaches employ ideas from model reduction theory to do direct inversion using reduced order forward models \cite{BoDrMaZa},\cite{DrMaThZa},\cite{BoDrMaZa2},\cite{BoDrMaZa3},\cite{DrMaZa} with impressive numerical results. 
These works have grown out of earlier work on spectrally matched finite difference grids \cite{DrKn}, where special grid steps were chosen so that the numerical solution matched functionals of the solution at the boundary in the spectral domain.  The spectrally matched grid steps depended very weakly on the medium, and this allowed for the direct extraction of the unknown coefficients \cite{BoDr}. 
There were some extensions of this to higher dimensions \cite{BoDrMa}, \cite{BoDrMaGu}, \cite{BoDrGuMa}, however, the grid itself was fundamentally one-dimensional, blocking the generation of truly higher dimensional direct inversion methods, except for special geometries.  The more recent works use a more general reduced order model approach  \cite{BoDrMaZa},\cite{BoDrMaZa2},\cite{BoDrMaZa3},\cite{DrMaThZa},\cite{DrMaZa} .

Here we show how one can obtain spectrally matched finite difference operators with well chosen localized Galerkin basis functions.  Using the basis functions, higher dimensional extensions are natural. This was essentially what was used in \cite{DrMaThZa},\cite{BoDrMaZa} in the time domain. In this work we consider the spectral domain, and we show explicitly how to view the reduced order models (ROMs) through the spectral Galerkin equivalence mentioned above. The Galerkin equivalence of spectrally matched grids was known {\it on the boundary} from \cite{DrMo}, however the basis in which the Galerkin model is the same as the finite difference operator everywhere was not known. 

A main advantage first observed in \cite{BoDrMaZa} is that, like the grid steps,  {\it the basis functions appear to depend only very weakly on the medium}. It is this crucial observation that allows one to use the ROM for inversion.
Furthermore, by viewing the reduced order model as embedded into the continuous problem and using the reference medium basis functions, good approximations of the internal solutions can be produced from the boundary data. We will demonstrate this with numerical examples in one and two dimensions. 

The paper is organized as follows.  In Section 2 we state the reconstruction problem in general. In Section 3 we consider a one dimensional version with one source. In Section 3.1 we follow the Loewner framework \cite{AnLeSa} and describe how one obtains a Galerkin model from the data directly, in Section 3.2 we describe the orthogonalization and prove the equivalence everywhere with spectrally matched grids. In Section 3.3 we present one dimensional numerical examples along with the internal solution generation/reconstruction algorithm.  In Section 4 we describe the same approach in higher dimensions; including the generation of the ROM from the data in Section 4.1 and the orthogonalization procedure in Section 4.2. Finally in Section 4.3 we present two dimensional numerical examples and the internal solution generation/reconstruction algorithm, and Section 5 contains some concluding remarks. 
\section{Problem statement.}
Consider the problem (\ref{2dprob1}) where we apply $K$  Neumann boundary functions $\{ g_1, \ldots,g_K \} $, with each $g=g_r\in H^{-1/2}(\partial \Omega)$ and denote by $u^r(x;\lambda)$ the solution to (\ref{2dprob1}) 
corresponding to spectral value $\lambda$ and source $g=g_r$. We read data in the form of the matrix valued transfer function as depending on the spectral parameter
$$ F_{rl}(\lambda) := \int_{\partial \Omega} u^r(x;\lambda) g_l(x)d\sigma_x . $$
If, for example, each Neumann function is an approximate delta function corresponding to a source/receiver location, 
for $r,l =1,\ldots,K$, the transfer function $F_{rl}(\lambda)$  corresponds to reading the solution coming from the $r$th "source"  at the $l$th "receiver".  
We take $m$ spectral values $\lambda= b_1,\ldots, b_m$, 
and define $$u^r_i := u^r(x; b_i), $$ the solution to (\ref{2dprob1}) at $\lambda=b_i$ with $g=g_r$. Now let us assume that we are in possession of boundary data of the form  
$$ F^i_{rl} := F_{rl}(b_i)= \int_{\partial \Omega} u_i^r g_l $$
and 
$$DF^i_{rl} := {d F_{rl}\over{d\lambda}}(\lambda)|_{\lambda=b_i} $$
 where the derivative here is with respect to $\lambda$. The inverse problem is to determine $q$ from the data \begin{equation} \label{data} \{ F^i_{rl}, DF^i_{rl} \} \ \ \ \mbox{for} \ \ 
 i=1,\dots, m\ \  \mbox{and}\ \  r,l=1,\ldots,K. \end{equation} Such data can be obtained from the time domain via a Fourier transform. 
\begin{remark} As we will see below, from this data one can find a ROM which matches the data exactly. However,  it is not at all crucial that the data be of the Hermite form (\ref{data}). In fact, matching point values of $F$ at $2m$ spectral points leads to a ROM with even better approximation properties and still yields a stable system \cite{DrKn}, \cite{DrMo}. Other forms of spectral data are also possible \cite{AnLeSa},\cite{DrMo}. However, we consider only the Hermite data (\ref{data}) in this work for simplicity. 
\end{remark}

\section{One dimensional problems.}
\subsection{One dimensional Galerkin model from spectral data.}
To begin, we follow the Loewner framework for model reduction \cite{AnLeSa}.
Consider solving the one dimensional, one source version of (\ref{2dprob1}),
\begin{eqnarray} \label{1d1side} - {d^2\over{dx^2}} u(x;\lambda) +q(x) u(x;\lambda)+\lambda u(x;\lambda) & = & 0 \ \ \mbox{for} \ \ x\in (0,1) \\   - {d\over{dx}} u(0;\lambda)  =  1 & & {d \over{dx}} u(1;\lambda) = 0 \nonumber 
\end{eqnarray}
where we read the data
$u(0;\lambda)$ and ${d\over{d\lambda}}u(0;\lambda)$ for $\lambda>0 $ as would be natural coming from the time domain. ( ln principle $\lambda$ could be anywhere in $\mathbb{C}$ away from the Neumann eigenvalues of the Schr\"odinger operator ).  One sees easily that (\ref{1d1side}) has the variational form: Find $u\in H^1((0,1))$ such that 
\begin{equation}\label{varform1d} \int_0^1  u^\prime \phi^\prime +\int_0^1 q u \phi + \lambda \int_0^1 u\phi = \phi(0)  \end{equation}
for all $\phi\in H^1((0,1))$. We note that by Sobolev embedding and the trace theorem the values of $\phi(0)$ are well defined for all test functions in $H^1$.  For any given value of $\lambda$, let us define our true spectral domain medium response by 
$$ F(\lambda)= u(0; \lambda ) $$
where $u(x,\lambda)$ is the solution to (\ref{varform1d}). 

Consider the exact solutions  $u_1,\ldots, u_m$  to (\ref{varform1d}) where $$u_j(x):= u(x, b_j).$$ These form the subspace $$V_m= \mbox{span} \{ u_1,\ldots, u_m\},$$ and we can consider the Galerkin solution $u_G$ to the variational problem: Find $u_G \in V_m$ such that \begin{equation}\label{Gvarform1d} \int_0^1  u_{G}^\prime \phi^\prime +\int_0^1 q u_{G} \phi + \lambda \int_0^1 u_{G}\phi = \phi(0)  \end{equation}
for any $\phi\in V_m$.  Searching for the unknown coefficients $\{ c_i\}$ for the solution $u_{G} =\Sigma_{i=1}^m c_i u_i$ and by setting $\phi= u_j$, one obtains the system
\begin{equation}\label{galROM} (S+\lambda M) \vec{c}= \vec{F} \end{equation}
where $\vec{F} = ( u_1(0), \dots, u_m(0))^\top = (F(b_1), \dots, F(b_m) )^\top $, the vector $\vec{c}=( c_1, \ldots c_m)^\top $, and
$M,S$ are the $m\times m$ mass and stiffness matrices given by 
$$M_{ij} = \int_0^1 u_i u_j $$
and
$$ S_{ij} = \int_0^1  u_i^\prime u_j^\prime +\int_0^1 q u_i u_j . $$
Using the variational formulation (\ref{Gvarform1d}) with $u_G=u_i$ and $\phi=u_j$ (recall $u_i$ is the exact solution for $\lambda=b_i$),   we obtain that 
\begin{equation} \label{varsys2} S_{ij} + b_i M_{ij} = F(b_j) \ \ \ \mbox{for all}\ i,j=1,\ldots, m. \end{equation}
Using this, if we are in possession of data $\{ F(b_j), F^\prime (b_j) : j=1,\ldots, m\}$, we can actually directly reconstruct $S$ and $M$. By reversing $i$ and $j$ and subtracting, for $i\neq j$ we have 
$$ (b_i-b_j) M_{ij}= F(b_j)- F(b_i) $$
or
\begin{equation} \label{Meq}  M_{ij}= { F(b_j)- F(b_i) \over{b_i-b_j}} \end{equation} 
where we have used the symmetry of $S$ and $M$. Imagine now that one were to have data $F(z)$ corresponding to solution $u_z$ for $z$ some spectral point close to $b_i$. Then by the above one has the formula 
$$\int_0^1 u_i u_z =  { F(z)- F(b_i) \over{b_i-z}}.$$
 Taking $z\rightarrow b_i$, we have that $u_z\rightarrow u_i$, so that 
\begin{equation} \label{Meq2} M_{ii}= \int_0^1 u_i^2 = - F^\prime (b_i) .\end{equation}
Multiplying (\ref{varsys2}) by $b_j$, reversing $i$ and $j$ and subtracting, we get 
$$ (b_j-b_i) S_{ij}= b_j F(b_j)- b_i F(b_i) $$
or \begin{equation} \label{Seq}  S_{ij}= {b_j F(b_j)- b_i F(b_i) \over{b_j-b_i}}.\end{equation} 
Again by taking some spectral point $z$ close to $b_i$ and letting $z\rightarrow b_i $ one obtains
\begin{equation} \label{Seq2} S_{ii}= (\lambda F)^\prime (b_i) .\end{equation}
Hence the formulas (\ref{Meq}),(\ref{Meq2}),(\ref{Seq}),(\ref{Seq2}), well known in the model reduction community \cite{AnLeSa}, provide the mass and stiffness matrices directly from the data. Similar formulas can be derived for other forms of spectral data. Together $S$ and $M$ provide a reduced order model for (\ref{1d1side}). 

The solution $u_G$ to (\ref{Gvarform1d}) also has the property that at $x=0$, as a function of $\lambda$ it matches the data \begin{equation}\label{hermform} \{ F(b_j), F^\prime (b_j) : j=1,\ldots, m\}\end{equation} exactly. We show this in the following Lemma, which was previously known \cite{AnLeSa}. 
\begin{lem}\label{implem}
Let $u_G(x;\lambda)$ be the Galerkin solution to (\ref{Gvarform1d}) for subspace $V_m$ generated by exact solutions corresponding to $\lambda= b_1,\dots, b_m$.  Let  $F_m$ given by \begin{equation}\label{ratform} F_m(\lambda)= \sum_{i=1}^m \frac{y_i^2}{\lambda-\theta_i} \end{equation} be the unique rational Hermite interpolant of form (\ref{ratform}) to $F(\lambda)$  matching the data (\ref{hermform}) with positive residues and negative poles. We then have 
$$u_G(0;\lambda) = F_m(\lambda) $$
for all $\lambda$. 
\end{lem}
\begin{proof}
First we show that the Galerkin solution yields a response at $x=0$ of the form (\ref{ratform}) for all $\lambda$. 
Here we use the continuous $L^2(0,1)$ inner product $\langle, \rangle$. Decomposing the solution $u_G$ into an orthonormal basis of Galerkin eigenpairs $ \{ w_i , \mu_i \} $,  
$$ u_G = \sum_{i=1}^m \langle u_G, w_i \rangle w_i ,$$ we have \begin{equation} \label{Gimpform} u_G(0;\lambda) = \sum_{i=1}^m \langle u_G, w_i \rangle w_i (0). \end{equation} 
The fact that $ \{ w_i , \mu_i \} $ is an eigenpair means that 
$$ \langle u_G, w_i\rangle = -{1\over{\mu_i}} \left( \int w_i^\prime u_G^\prime +\int q w_i u_G \right) $$
which from the variational formulation for $u_G$ gives 
$$ \langle u_G, w_i\rangle ={1\over{\mu_i}} \langle w_i, \lambda u_G \rangle -{1\over{\mu_i}} w_i(0) .$$ 
Solving for $ \langle u_G, w_i\rangle$ we have
$$ \langle u_G, w_i\rangle ={w_i(0)\over{\lambda-\mu_i}} $$ which from (\ref{Gimpform}) yields 
$$ F_G(\lambda) := u_G(0;\lambda) = \sum_{i=1}^m  {w^2_i (0)\over{\lambda -\mu_i}}.$$
Hence $u_G(0;\lambda)$ has the same form as  $F_m(\lambda)$ in  (\ref{ratform}), with positive residues and negative poles. Recall that $F_m$ is the rational Hermite interpolant to $F$ at the points $\lambda= b_1,\ldots, b_m $.  Now, since the trial space for the Galerkin solution contains the exact solutions for $\lambda = b_1\ldots b_m $, $u_G$ will be the exact solution at those spectral points, which means that at $\lambda=b_i$,
$$u_G(0;b_i) = F(b_i) = F_m(b_i) .$$
We claim $F_G(\lambda):=u_G(0;\lambda)$ also matches the derivatives at the spectral points as a function of $\lambda$.  The Galerkin solution $u_z:=u_G(0,z)$ for some $\lambda=z$ near $b_i$ can be used as a test function for $u_i$, so that one has  
$$ F_G(z) = u_z(0)= \int u_i^\prime u_z^\prime +\int q u_i u_z + b_i \langle u_i, u_z\rangle .$$
We can also use the variational formulation for this same Galerkin solution $u_z$ for $\lambda= z$ with $u_i$ as a test function to obtain
$$F_G(b_i)= u_i(0) =\int u_z^\prime u_i^\prime +\int q u_z u_i + z \langle u_z, u_i\rangle ,$$
and by subtracting them we have
$$ F_G(z)- F_G(b_i) = (b_i-z)  \langle u_z, u_i\rangle.$$
Taking the limit as $z\rightarrow b_i$ one obtains
$$F_G^\prime(b_i) =- \lim_{z\rightarrow b_i} \langle u_z, u_i\rangle = -\langle u_i, u_i\rangle = F^\prime (b_i) $$ 
from (\ref{Meq2}).  Hence $F_G(\lambda)= u_G(0;\lambda)$ is the rational Hermite interpolant to $F$ at the given spectral points, and so by uniqueness of this Pad\'e approximation, 
$$ u_G(0;\lambda) = F_m(\lambda) $$
for all $\lambda$. 
\end{proof}

Spectral Galerkin models such as the above are difficult to work with since they tend to yield full matrices. In particular for inverse problems, it is not obvious how one may extract the coefficients, since we wouldn't know the basis functions $\{ u_i \} $ internally from the data. It is for these reasons we follow along the lines of \cite{BoDrMaZa},\cite{DrMaZa} and do a Lanczos orthogonalization to construct a new orthonormal basis for $V_m$  in which the operator $S$ becomes tridiagonal. In the new basis, $M$ is the identity from the orthonormality. 

\subsection{Lanzcos orthogonalization and equivalence with spectrally matched finite difference grids.}
The original idea of spectrally matched grids \cite{DrKn} was that, given receiver data in the spectral domain, choose a rational approximation for this data, then find the three point finite difference stencil whose forward solution at the receiver matches exactly this rational approximation. 
This yielded a very special nonuniform grid and spectral convergence at the receiver. Higher order convergence for the forward model was sacrificed elsewhere in domain, but this was irrelevant for inversion. It was also shown at the time that the receiver response of the finite difference scheme was equivalent to that of a spectral Galerkin method, for various types of spectral data \cite{DrMo}.  Just as above we saw how to generate a Galerkin model from data, one can generate a finite difference model from the same data. 

Suppose we read the data (\ref{hermform}) as before. 
Forgetting for now about the Galerkin method, there is a unique rational approximation $F_m(\lambda)$ to $F(\lambda)$ of the form  (\ref{ratform})
with positive residues and negative poles which gives this Hermite interpolant to $F(\lambda)$, matching the data (\ref{hermform}). From these positive residues and negative poles, by using what is essentially the Lanczos algorithm, one can find a continued fraction form for $F_m(\lambda)$,
\begin{equation*}
F_m(\lambda) = \cfrac{1}{\hat{\gamma}_1\lambda +\cfrac{1}{\gamma_1 +\cfrac{1}{\hat{\gamma}_2\lambda +\cdots\cfrac{1}{\gamma_{m-1}+\cfrac{1}{\hat{\gamma}_m\lambda  }}}}} .
\end{equation*}
Then from there one can extract $\{ \gamma_i, \hat{\gamma}_i \} $ which define a three-point staggered difference scheme with tridiagonal matrix $L_m$  for which the approximated solution at $x=0$ is exactly $F_m(\lambda)$  \cite{DrKn}.  That is, the rational approximation $F_m(\lambda)$ uniquely determines positive $\gamma_j, \hat{\gamma}_j$, such that solving the finite difference scheme 
\begin{eqnarray}\label{finitediff} -{1\over{\hat{\gamma}_j}} \left( {U_{j+1}-U_j\over{\gamma_j}}-  { U_j -U_{j-1} \over{\gamma_{j-1}}}\right) +\lambda U_j &=& 0 \ \ \ \mbox{for}\ \ j=1,\ldots m \\  - {U_1- U_0\over{\gamma_0}} =1, & &  {U_{m+1}-U_m\over{\gamma_m}}=0 \nonumber\end{eqnarray}
yields $$U_1= F_m(\lambda).$$ Note that $U_0$ and $U_{m+1}$ are ghost points only (and $\gamma_0,\gamma_m$ ghost grid steps) and the nonzero Neumann condition on the left yields a nonzero right hand side in the first component. That is, this yields an $m$ dimensional matrix system of the form
$$(  L_m +\lambda I )\vec{U} = {\vec{e}_1\over{{\hat{\gamma}_1}}}$$
for the unknown $U_1,\ldots, U_m$, with entries of $L_m$ given by (\ref{finitediff}) in terms of the $\gamma_j$ and $\hat{\gamma}_j$.  
\begin{remark} When $q=0$, the $\gamma_j, \hat{\gamma}_j$ are primary and dual finite difference grid steps.  For nonzero $q$, finite difference scheme (\ref{finitediff}) can be transformed to a discrete Schr\"odinger form with a discrete Liouville transform \cite{BoGuMa}. For simplicity, we won't do that here, since we don't use the finite difference framework for reconstructions. \end{remark}

Consider again approximating the solution to  (\ref{1d1side}) by a Galerkin method with the finite dimensional subspace $V_m$, generated by the
exact solutions $\{ u_i\}$ at the spectral points $\lambda = b_i$. 
Since the finite difference grid does not easily extend to higher dimensions, we would rather work with Galerkin systems. Recall the Galerkin system (\ref{galROM}) 
$$(S+\lambda M) \vec{c}= \vec{F} $$
where $\vec{F} = (F(b_1), \dots, F(b_m) )^\top $ and the vector $\vec{c}=( c_1, \ldots c_m)^\top $ 
contains the coefficients for solution $$ u_G= \sum c_i u_i .$$
Unfortunately, unlike the finite difference model, the stiffness and mass matrices $S$ and $M$ are full.
However, one can in fact find a new basis for $V_m$ in which the Galerkin system is tridiagonal, and {\it exactly everywhere} the same as the symmetrization of the finite difference system (\ref{finitediff}). 
To do this we will do Lanczos orthogonalization. Define $\delta_{m}\in V_m$ to be the projection of the delta function onto $V_m$,
that is, the unique element of $V_m$ which satisfies \begin{equation}\langle \delta_m , w\rangle = w(0) \ \ \ \mbox{for all} \ \ \ w\in V_m, \nonumber \end{equation} 
identified with the coefficient vector $\vec{d} = (d_1,d_2,\ldots d_m )$. Note that $\vec{d}=M^{-1}\vec{F}.$ Consider now the basis where we let $A=M^{-1}S$ and 
$$B=\{  \vec{d}, A\vec{d}, A^2\vec{d}, \ldots, A^{m-1}\vec{d} \}.$$
and we orthogonalize using Gram-Schmidt with respect to the mass matrix $M$ inner product 
$$\langle \vec{x},\vec{y} \rangle_M := \langle M\vec{x},\vec{y}\rangle .$$ This will yield continuous $L^2$ orthogonality of the basis functions. Note that $A$ is symmetric with respect to this inner product. 
We get new orthogonalized basis for $V_m$: 
\begin{equation}\label{orthogbasis} V_m=\mbox{span}\{ \hat{u}_1(x), \hat{u}_2(x),\ldots \hat{u}_m(x) \}. \end{equation}
Due to the orthogonalization, the new basis is somewhat localized, more refined near zero and coarsening out towards one, as spectrally matched finite difference grid steps do. See an example in Figure \ref{basisfig} for $m=6$. Furthermore, $A$ is tridiagonal in the new basis, and since $M$ becomes the identity, this means that $S$ is also tridiagonal in the new basis.

\begin{figure}
\centering
\includegraphics[scale=.5]{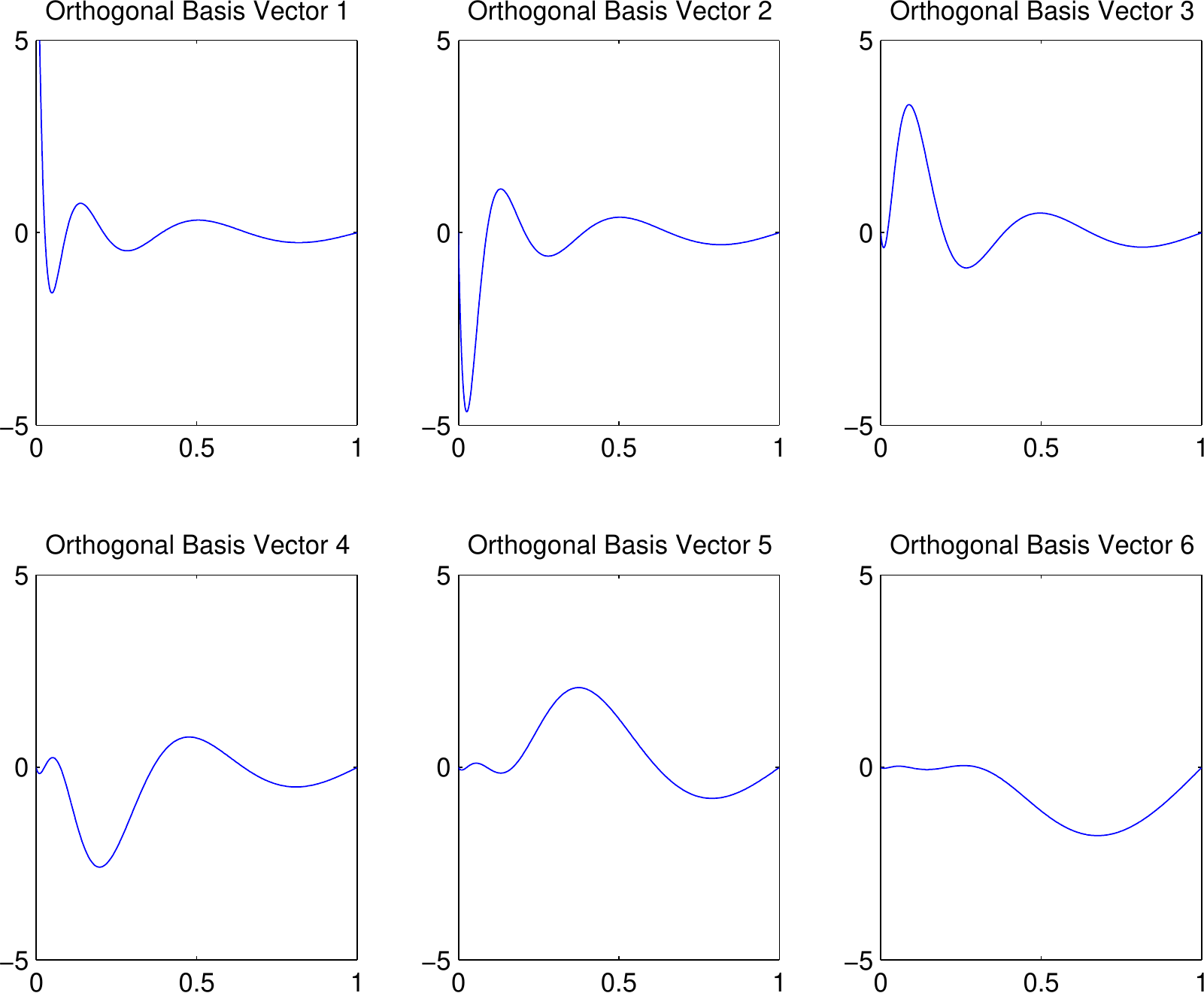}
\caption{New basis $\hat{u}_1,\ldots,\hat{u}_m$ in which the Galerkin system becomes a tridiagonal finite difference system. Note the localization of the basis functions, with peaks refining near zero and coarsening out towards 1.}
\label{basisfig}\end{figure}

\begin{thm} If we change the basis to the orthogonalized (\ref{orthogbasis}) and form the Galerkin system in this new basis
\begin{equation} (\hat{S}+\lambda \hat{M})\vec{\hat{c}}=\vec{\hat{F}}\label{galsyshat} \end{equation}
(with $\hat{M}=I$) to solve for $\vec{\hat{c}}$, this is the symmetrization of the finite difference system (\ref{finitediff}) for $\vec{U}$. In particular,  
 $$u_G =\sum_{i=1}^m \sqrt{\hat{\gamma_i}} U_i \hat{u}_i(x).$$
\end{thm}
\begin{proof}
We first symmetrize (\ref{finitediff}) by setting $$V_j =\sqrt{\hat{\gamma}_j} U_j$$ and and then multiplying each row $i$ by
 $\sqrt{\hat{\gamma}_i}$ in the resulting equations for the $\{ V_i \} $. 
The resulting system is in the form $$( L +\lambda I)\vec{V}= {1\over{\sqrt{\hat{\gamma}_1}}}\vec{e}_1,$$ where $L$ is now symmetric.  We recall that for the given impedance of the form (\ref{ratform}) (now multiplied by $\sqrt{\hat{\gamma}_1}$), the symmetric tridiagonal $L$ is uniquely determined.  

\begin{figure}
\centering
\includegraphics[scale=.4]{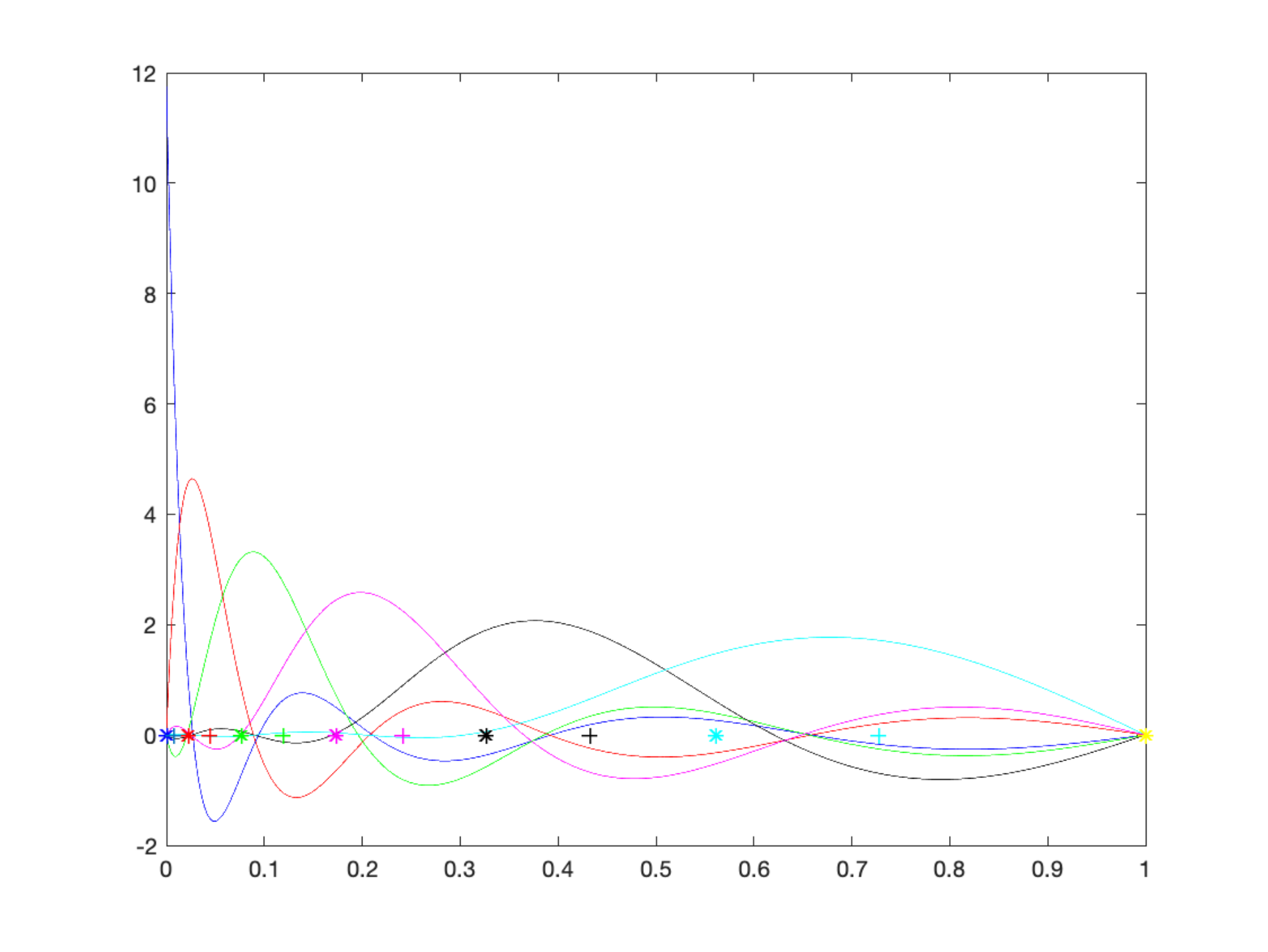}
\caption{Orthogonalized basis plotted with its corresponding finite difference grid. Stars $\ast$ correspond to primary grid points and plus signs $+$ to dual grid points. Primary grid points are color coded with corresponding basis functions. Note the masses of the basis functions are concentrated between each primary point and the next dual point. }
\label{basisandgrid}
\end{figure}

\begin{figure}[t!]
\centering
\includegraphics[scale=.3]{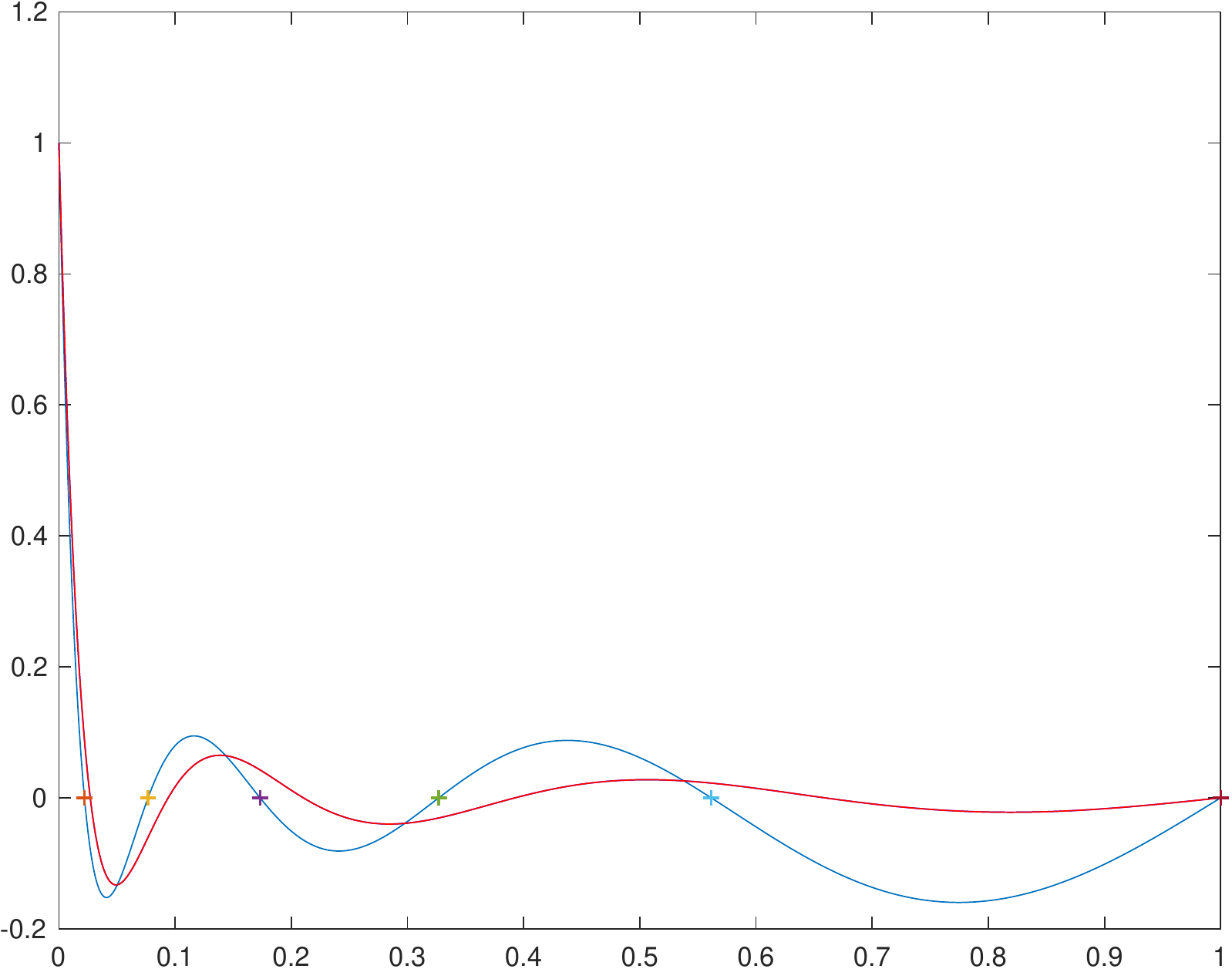}
\includegraphics[scale=.3]{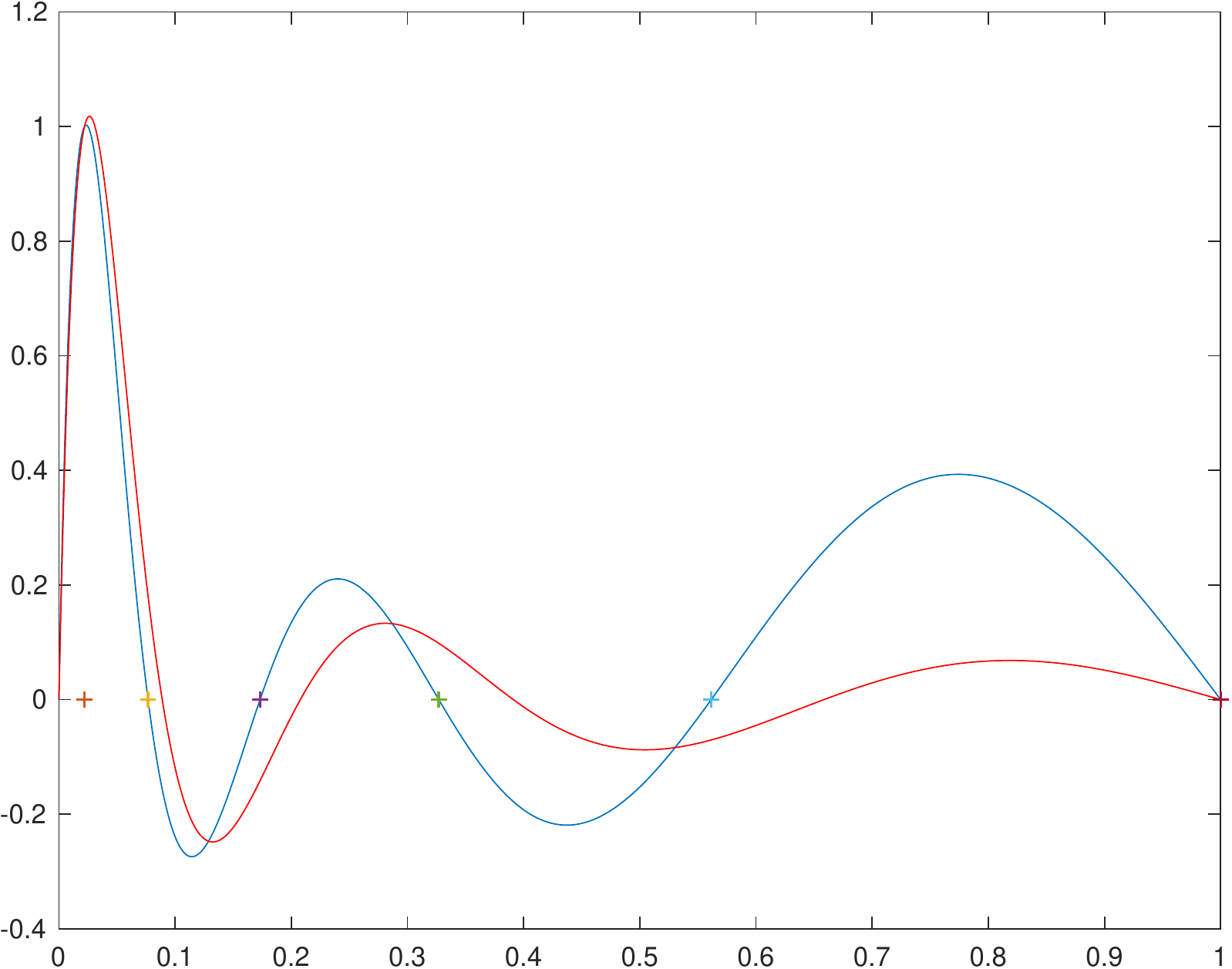}
\includegraphics[scale=.3]{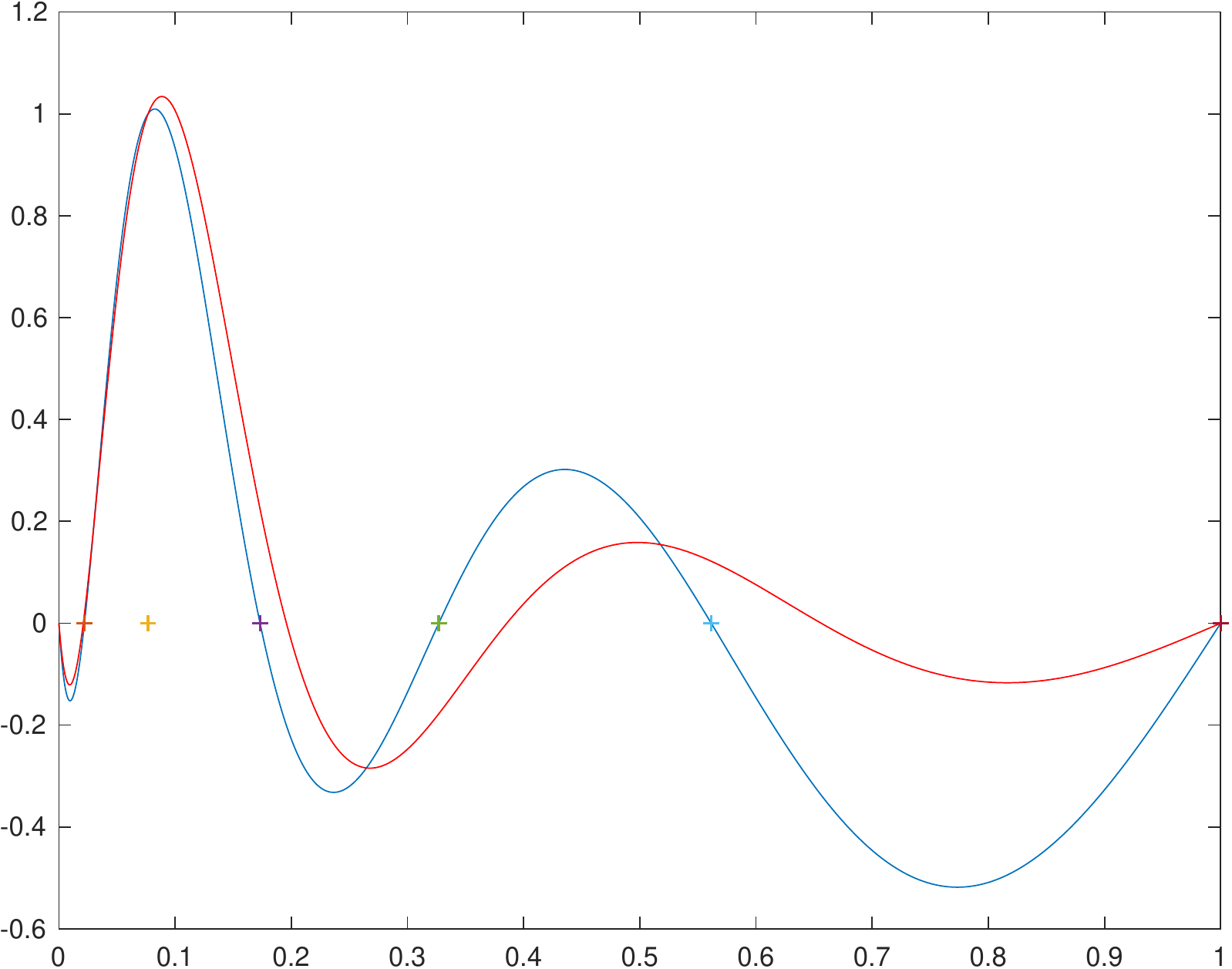}
\includegraphics[scale=.3]{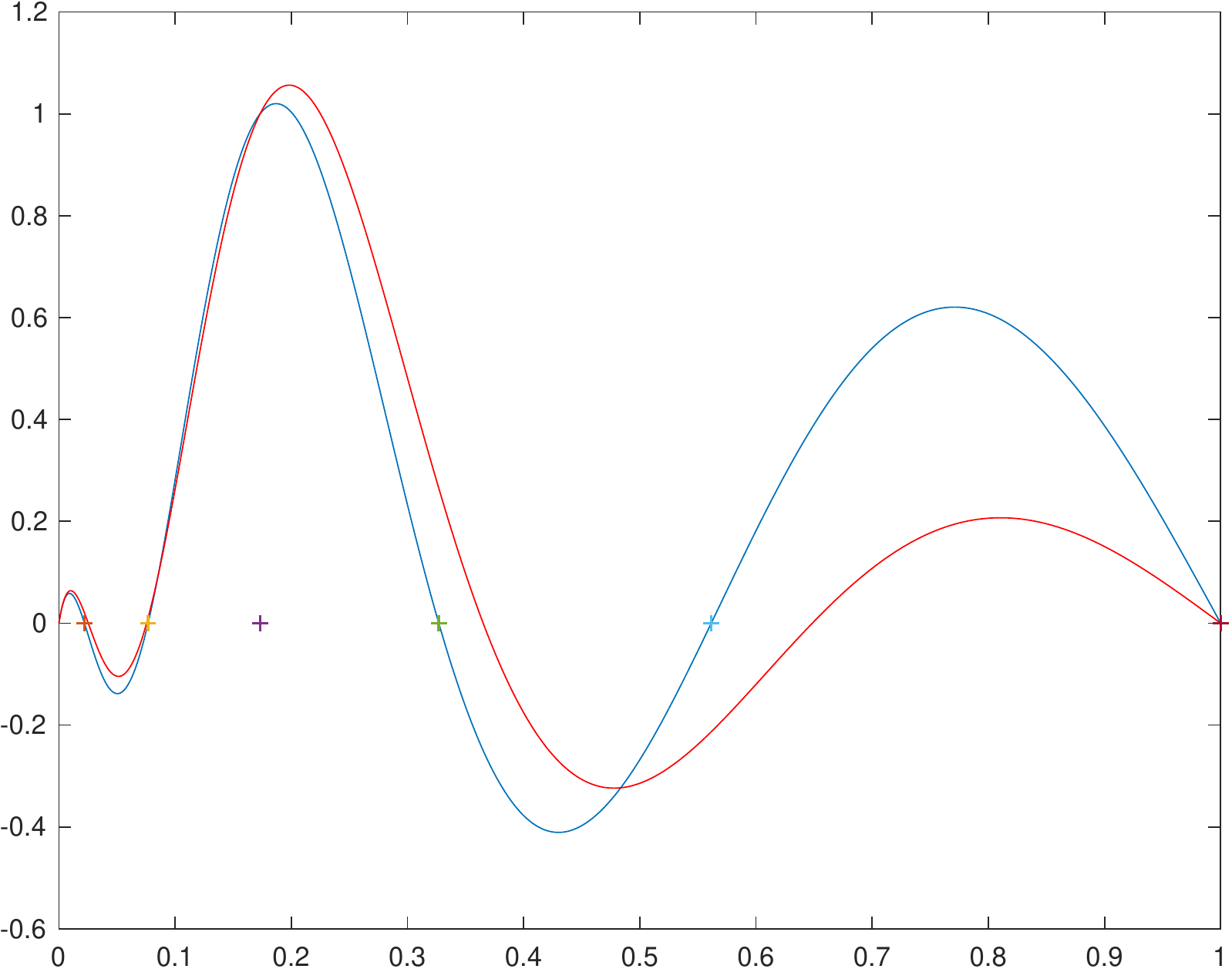}
\includegraphics[scale=.3]{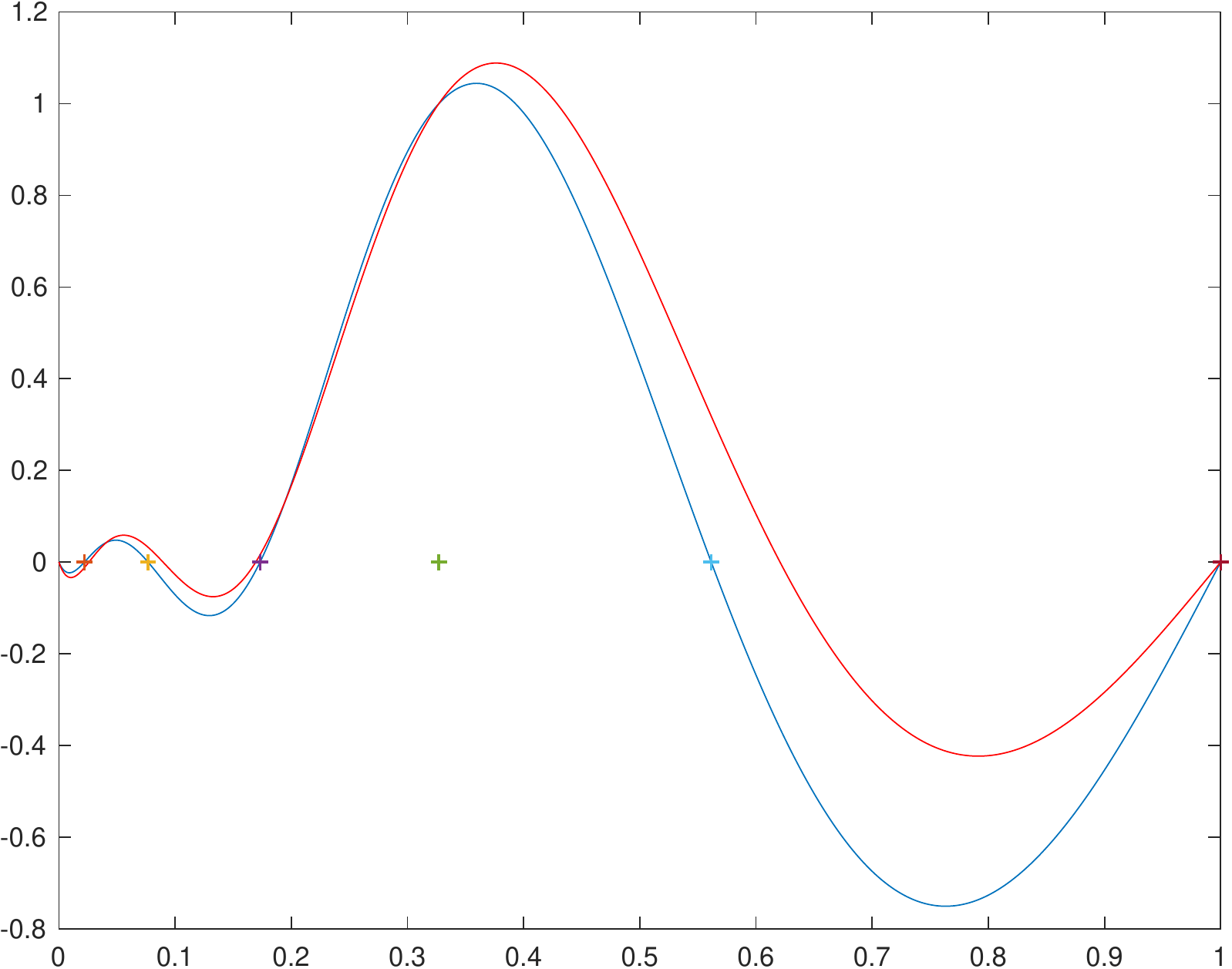}
\includegraphics[scale=.3]{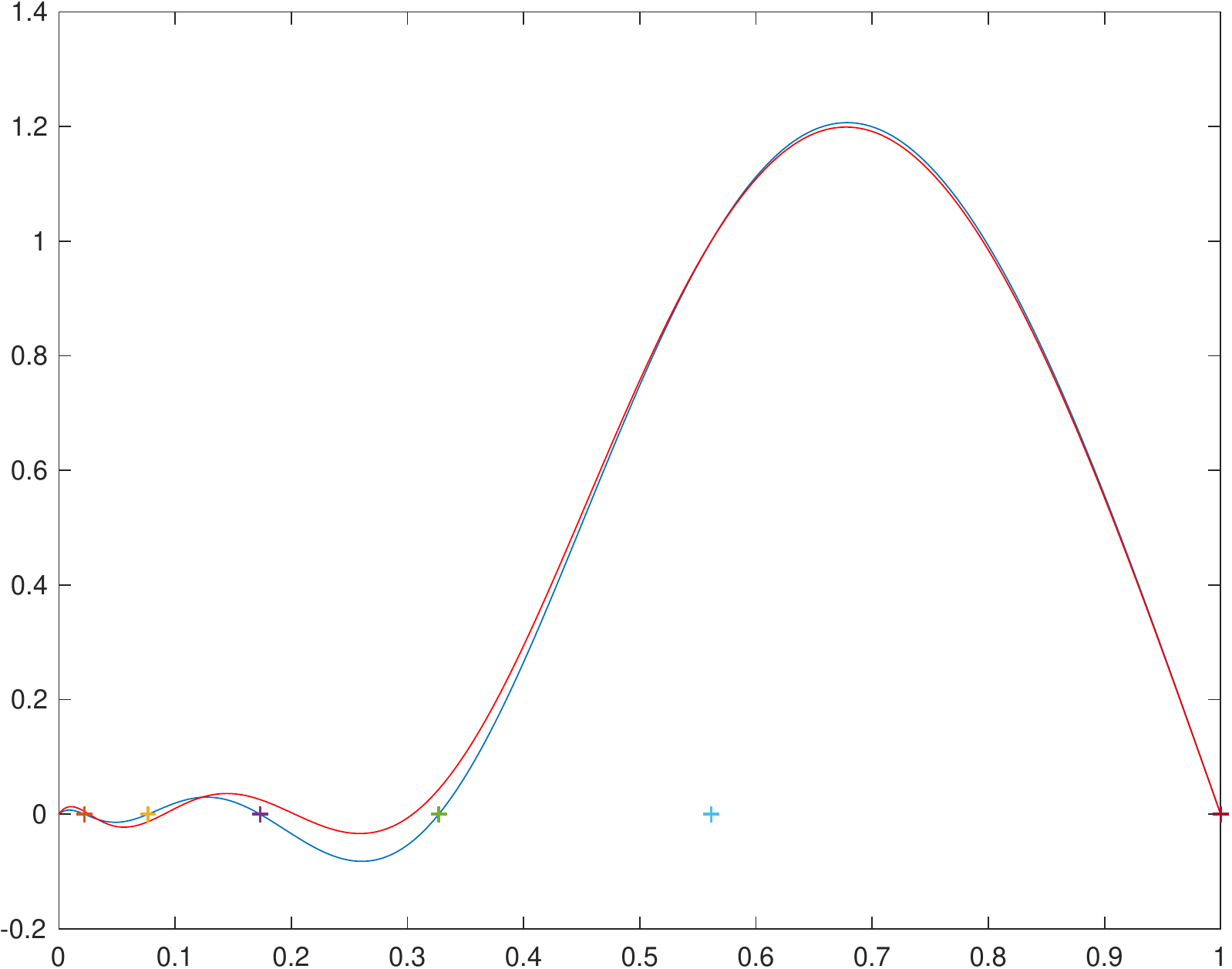}
\caption{Red: Orthogonalized basis normalized so that $\hat{u}_i(x_i)=1$ where $x_i$ is the $i$th primary grid point.  Blue: Nodal basis $\{ \phi_i \}$  of $V_m$ satisfying $\phi_i(x_j)=\delta_{ij}$. Note that the Lanczos orthogonalized basis resembles the nodal one, but has more localization.  }
\label{basislagrange}
\end{figure}

Consider now the system (\ref{galsyshat}) in the new basis. 
 The mass matrix $\hat{M}=I$ will be identity due to orthogonality. The stiffness matrix $\hat{S}$ will be tridiagonal due to the fact that this is a Lanczos process with polynomials in $A$, ($A$ becomes tridiagonal and therefore $\hat{S}$).  Furthermore, since $\hat{u}_1$ is the projection of the delta function, it will be nonzero at zero, but from orthogonality of the new basis functions we have that $$ \langle \hat{u}_1 , \hat{u}_i \rangle = \hat{u}_i (0) = 0$$ for all $i> 1$. Hence the right hand side $\vec{\hat{F}}$ is nonzero only in the first component and is therefore equal to $\hat{u}_1(0)\vec{e}_1$. This means that the system (\ref{galsyshat}) is of the form  $$( \hat{S} +\lambda I)\vec{\hat{c}}= \hat{u}_1(0)\vec{e}_1$$  for $\hat{S}$ tridiagonal symmetric. We also know from Lemma \ref{implem}  and the construction of the system (\ref{finitediff}) that $ u_G(0;\lambda) = F_m(\lambda) = U_1 $ for all $\lambda$, which means that $$\hat{c}_1 \hat{u}_1(0) = {1\over{\sqrt{\hat{\gamma}_1}}} V_1$$ for all $\lambda$.  By decomposing $\vec{\hat{c}}$  into the normalized eigenpairs of $\hat{S}$ as was done in the proof of Lemma \ref{implem} except now using the standard Euclidean inner product in $\mathbb{R}^m$, one can calculate that 
 $$ \hat{c}_1 = \hat{u}_1(0) \sum_{i=1}^m {z^2_i\over{\lambda-\theta_i}} $$
 for $\theta_i$ the eigenvalues and $z^2_i$ the squares of the first components of the normalized eigenvectors of $\hat{S}$.  Similarly
  $$ V_1 =  {1\over{\sqrt{\hat{\gamma}_1}}} \sum_{i=1}^m {y^2_i\over{\lambda-\mu_i}} $$
  where $\mu_i$ the eigenvalues and $y^2_i$ the squares of the first components of the normalized eigenvectors of $L$. 
 By the above impedance equality for all $\lambda$ and the normalization of the residues, one must have that the poles and residues are equal (and that the multipliers from right hand sides are the same). Therefore $L$ and $\hat{S}$ are the same by the uniqueness of the inverse eigenvalue problem, and $\vec{\hat{c}}= \vec{V}$, from which the result follows. 
\end{proof}
\begin{remark} From the above we also found that $$ \hat{u}_1(0) = {1\over{\sqrt{\hat{\gamma}_1}}}  ,$$
which coincides with $\hat{u}_1$ being a projection of the delta function. \end{remark}
 That is, the solution components $U_j$ of the difference scheme (\ref{finitediff})
can be interpreted as coefficients of the spectrally converging Galerkin solution with respect to the orthogonal basis (\ref{orthogbasis}). See Figure \ref{basisandgrid} for an example of an orthogonalized basis with its equivalent staggered finite difference grid. Note that the masses of the basis functions are concentrated between the primary grid steps and the next dual grid step, and they spread out in the same way that the grid steps coarsen. In Figure \ref{basislagrange}, we normalize each function in the basis to be equal to $1$ at its corresponding primary grid point, and plot it against the nodal basis for the same space. (The nodal basis is the unique basis $\{ \phi_i \}$ of $V_m$ satisfying $\phi_i(x_j)=\delta_{ij}$.) Note that our orthogonal basis resembles the nodal one, but is less oscillatory away from the corresponding grid point. In what follows below, we propose an inversion method using the Lanczos orthogonalized basis, which makes no use of the finite difference grid steps, and is easily generalizable to other geometries.

\subsection{Internal data generation and inversion in one dimension.}
Consider solving
\begin{eqnarray} \label{1d1side2}  -u^{\prime\prime} +q u +\lambda u & = & 0 \ \ \mbox{ on } \ \ \ (0,1) \\   - u^\prime (0) & = & 1 \nonumber \\ u(1) &=& 0 \nonumber 
\end{eqnarray}
as a perturbation of the corresponding reference problem 
\begin{eqnarray} \label{1d1sideref}  -u_0^{\prime\prime} +\lambda u_0 & = & 0 \ \ \mbox{ on } \ \ \ (0,1) \\   -u_0^\prime (0) & = & 1 \nonumber \\ u_0(1) &=& 0. \nonumber 
\end{eqnarray}
Here we use the reference medium $q_0=0$, however, any known reference medium should work the same. We note also that one can have either a Dirichlet or Neumann condition on the right endpoint $x=1$. The one dimensional numerical experiments below were done with a Dirichlet condition, a Neumann condition yielded similar results.

\noindent
\underline{\bf Algorithm}
\begin{enumerate}
\item Read data (here synthetically generated) $F(b_i)$, $F^\prime (b_i)$ for some positive $\lambda= b_1,\ldots b_m$ for the perturbed problem (\ref{1d1side2}), and compute all of the corresponding solutions $u_i^0$ for $i=1,\ldots,m$ for the reference problem.  

\item Use (\ref{Meq}, \ref{Meq2}, \ref{Seq}, \ref{Seq2}) to generate the $m\times m$ stiffness and mass matrices $S$ and $M$ for Galerkin system for the perturbed problem, and similarly compute $S_0$, $M_0$ for the reference problem. 

\item Compute projections of the delta function at zero onto $V_m= \langle u_i \rangle $ by solving for $\vec{d}$ in $$M\vec{d} = \vec{F} $$
where components $F_j=F(b_j)$.  These are the coefficients $d_j$ of projected $\delta$ in the $\{ u_i\}  $ basis for $V_m$. Similarly, we calculate $\vec{d}^0$ by solving $$ M_0\vec{d}^0 =\vec{F}^0,$$ where $(F^0)_j=u^0_j(0)$ is the reference data.  This means that $\{ (d^0)_j \}$ are the coefficients of the projected delta function in the reference basis $\{  u^0_i \} $ for the space  $ V_m^0 = \langle u_i^0 \rangle $. 

\item Set $A=M^{-1} S$, and note that $A$ is symmetric with respect to the inner product $ \langle \vec{x}, \vec{y}\rangle_{M} := M\vec{x}\cdot \vec{y}$, which is same the $L^2$ inner product on the corresponding continuous functions. Perform Lanczos orthogonalization with respect to this inner product on the basis 
$$\{  \vec{d}, A\vec{d}, A^2\vec{d}, \ldots, A^{m-1}\vec{d} \}$$
and let  $\hat{S}$,  $\hat{M}=I$ be the Galerkin stiffness and mass matrices in this new basis $ \{ \hat{u}_i\} $. Note that $\hat{S}$ is tridiagonal.

\item Similarly, set $A_0=M_0^{-1} S_0$ and perform Lanczos orthogonalization with respect to the $M_0$ inner product on the basis 
$$\{  \vec{d}_0, A_0\vec{d}_0, A_0^2\vec{d}_0, \ldots, A_0^{m-1}\vec{d}_0 \}$$ 
to obtain the corresponding orthognalized basis $\{ \hat{u}_i^0\} $ for $V_m^0$.

\item Now choose any spectral value $\lambda$ and solve  $ c= (A+\lambda I)^{-1} \vec{\hat{F}}  $ , so that $u_G=\sum c_i \hat{u}_i $ is the Galerkin approximation in $V_m$ to the solution of (\ref{1d1side2}) . Since we don't know $\hat{u}_i $ from the data, approximate $ u_G $ by $$ \tilde{u} = \sum c_i \hat{u}_i^0.$$

\item Compute  $(\tilde{u}^{\prime\prime} -\lambda \tilde{u})/\tilde{u} \approx q $; or reserve $\tilde{u}$ for another use. 

\end{enumerate}

\begin{figure}
\centering
\includegraphics[scale=.3]{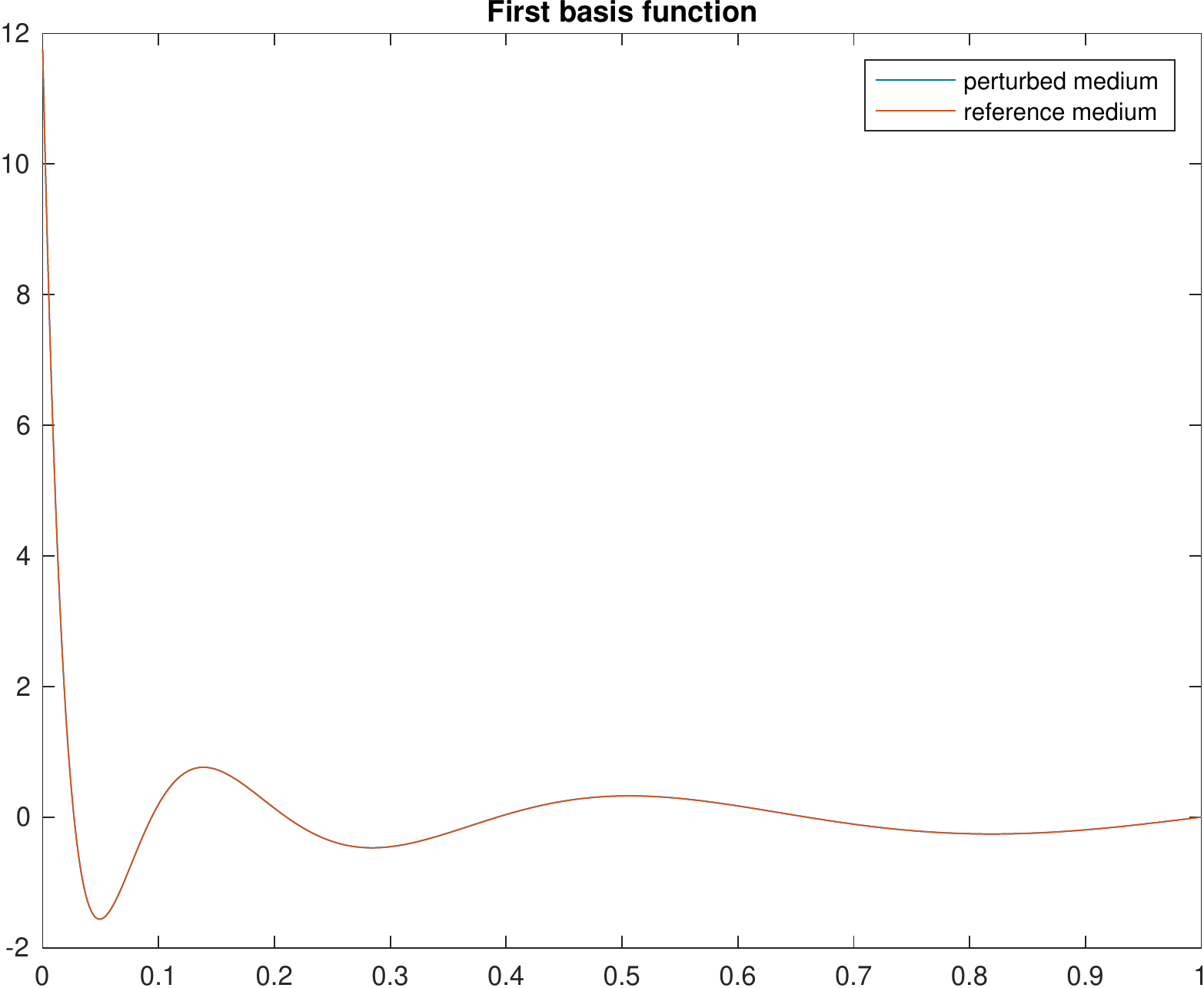}
\includegraphics[scale=.3]{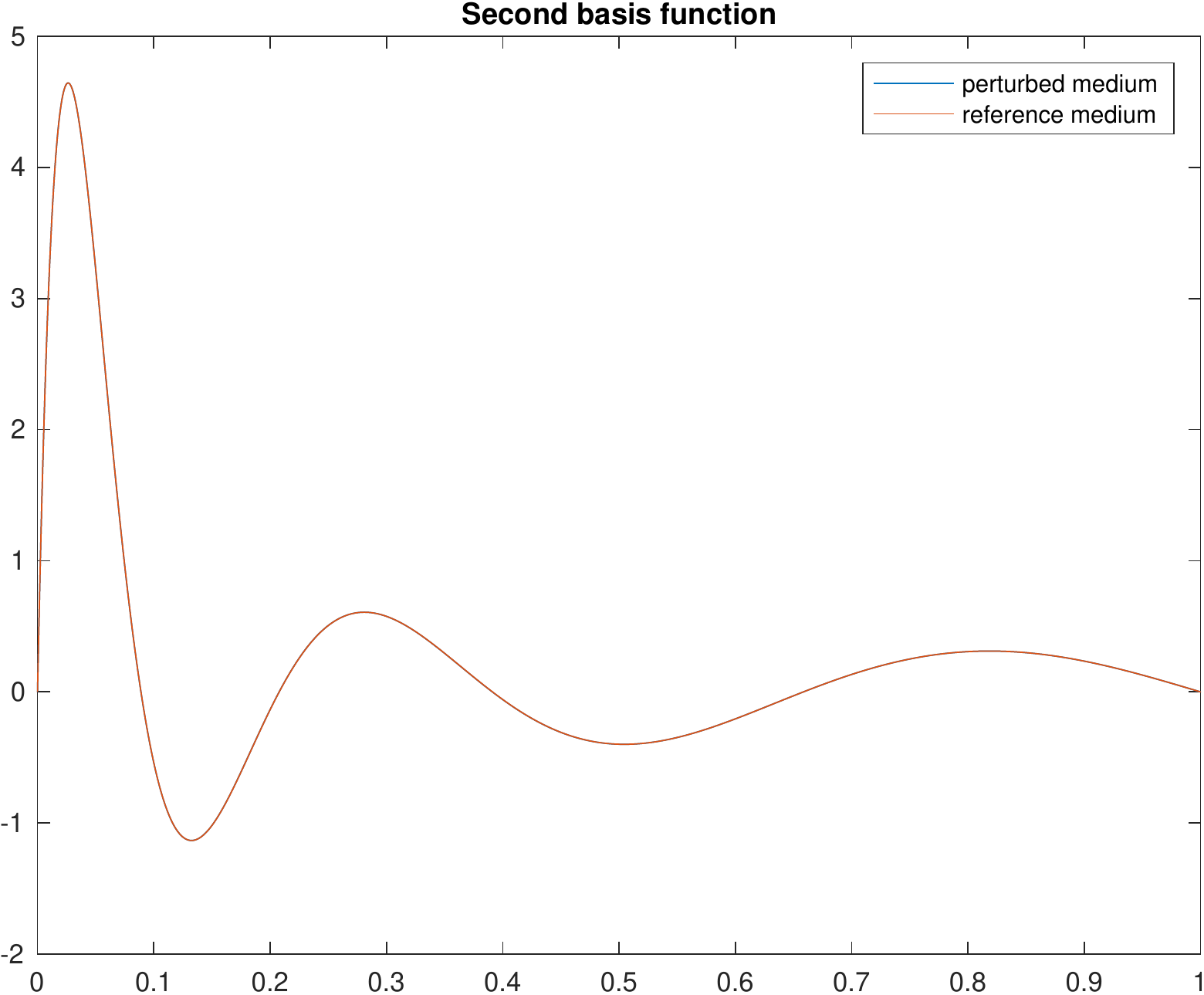} \\
\includegraphics[scale=.3]{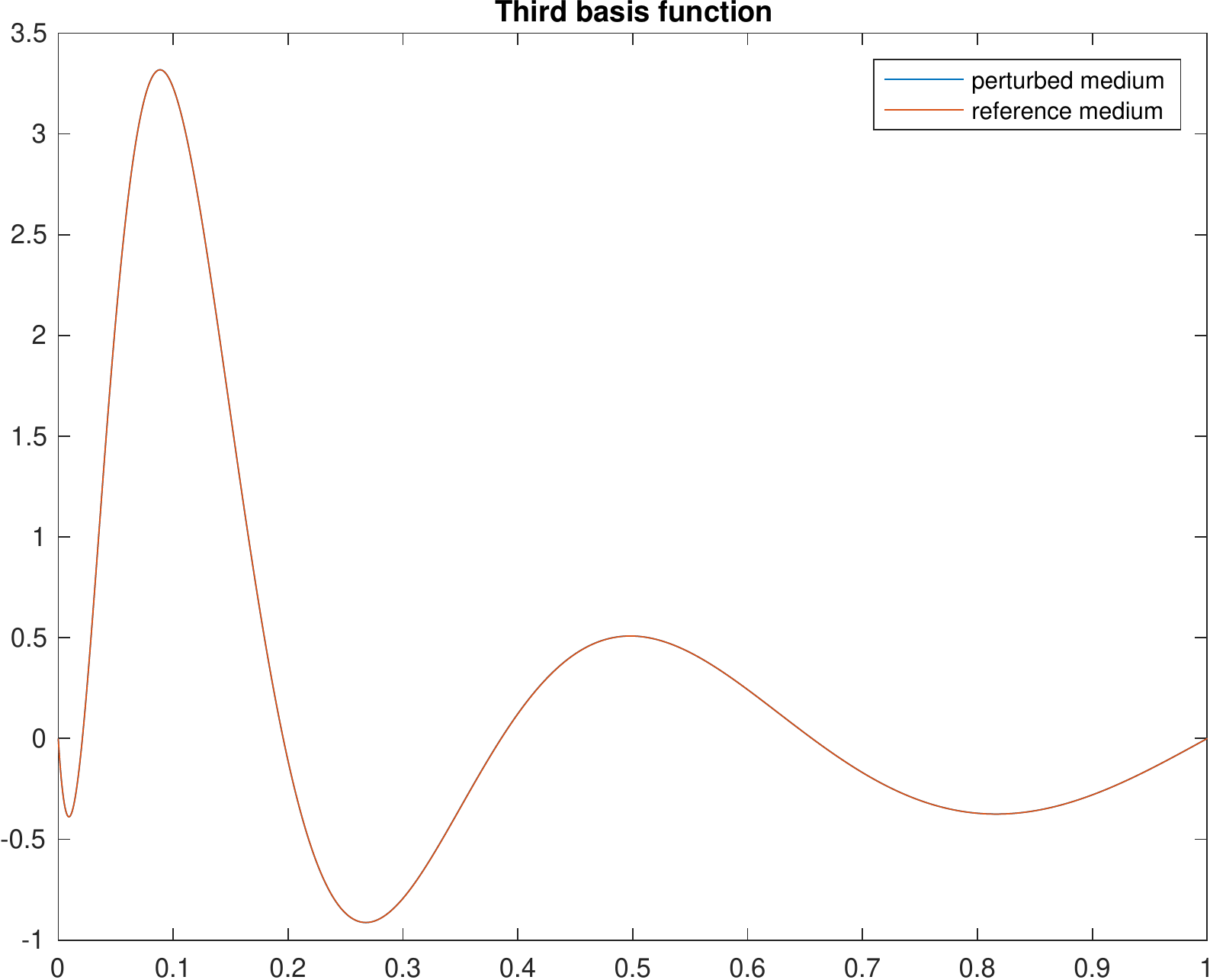}
\includegraphics[scale=.3]{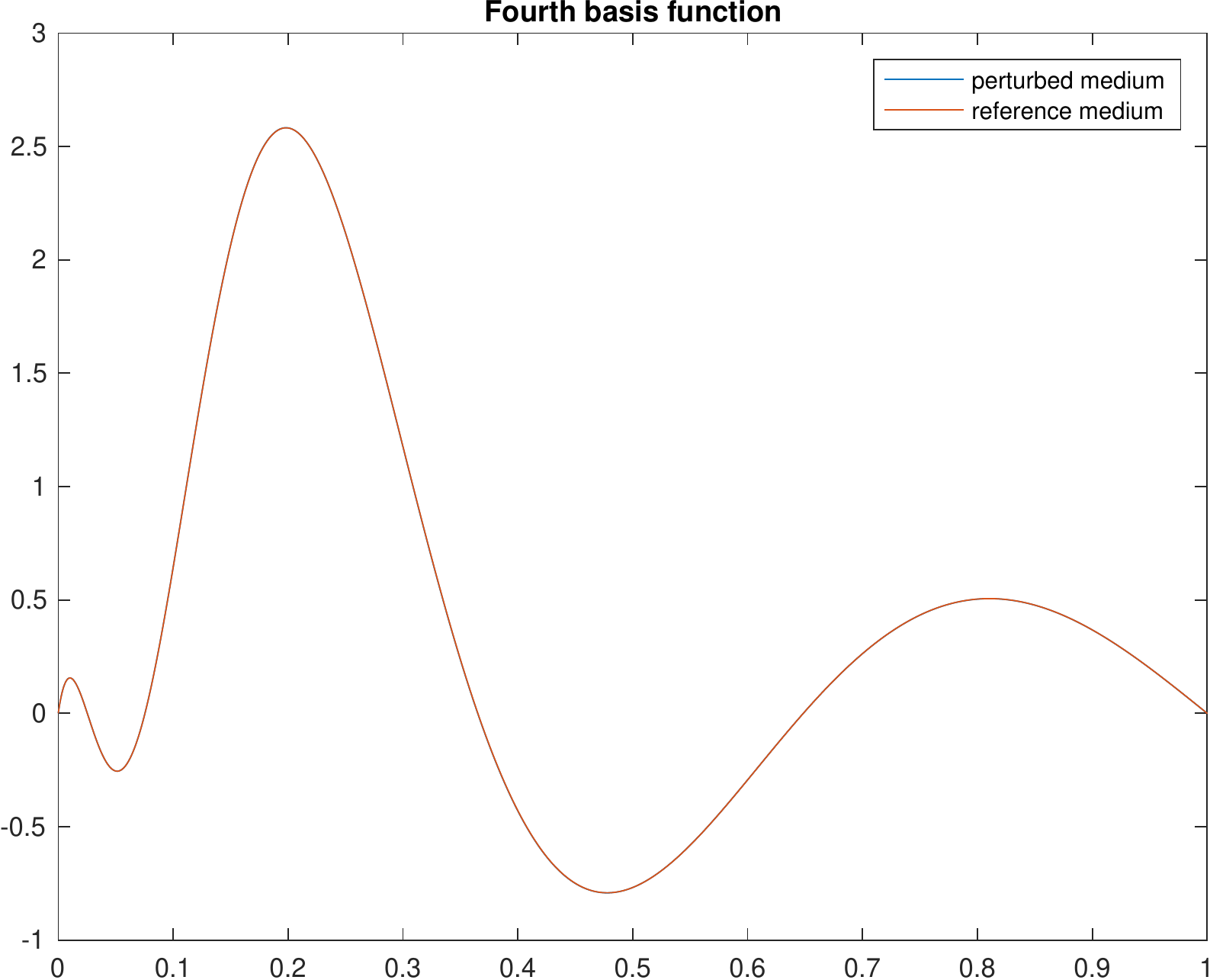} \\
\includegraphics[scale=.3]{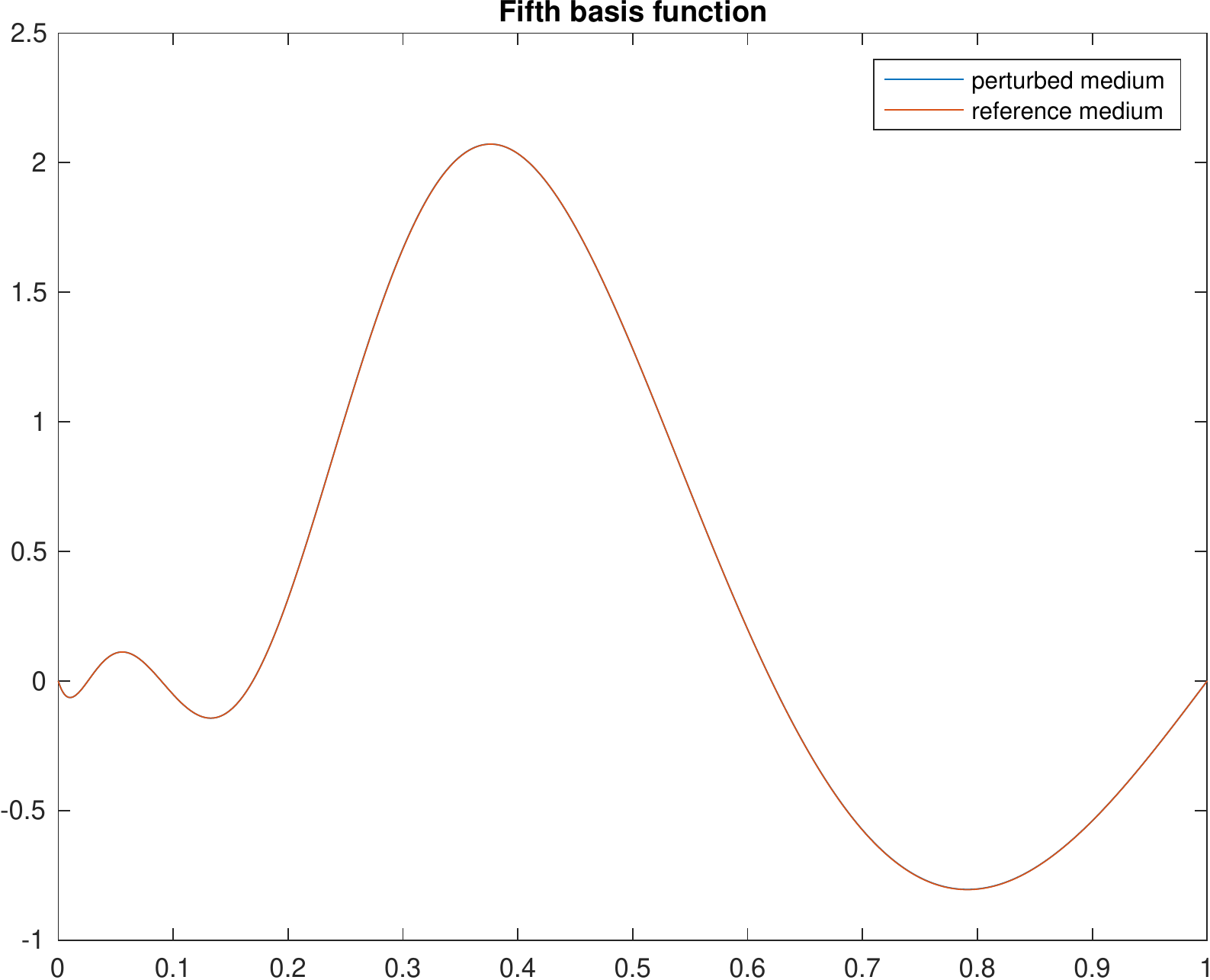}
\includegraphics[scale=.3]{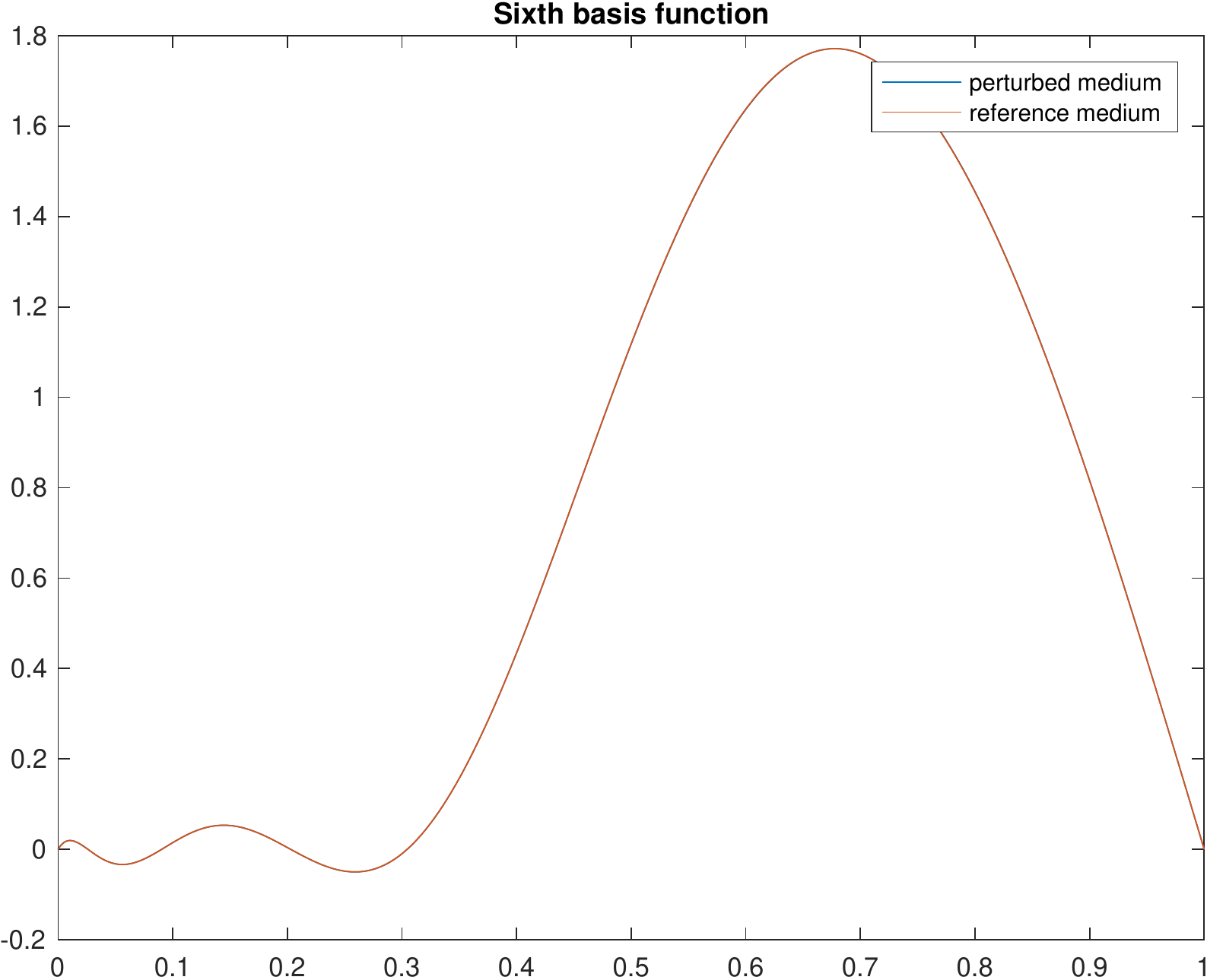} 
\caption{ Weak dependence of orthogonalized bases on $q$ . Orthogonalized basis functions for reference and perturbed medium visually coincide, here $m=6$. }
\label{basis}\end{figure}

%

\begin{figure} 
\centering
\includegraphics[scale=.34]{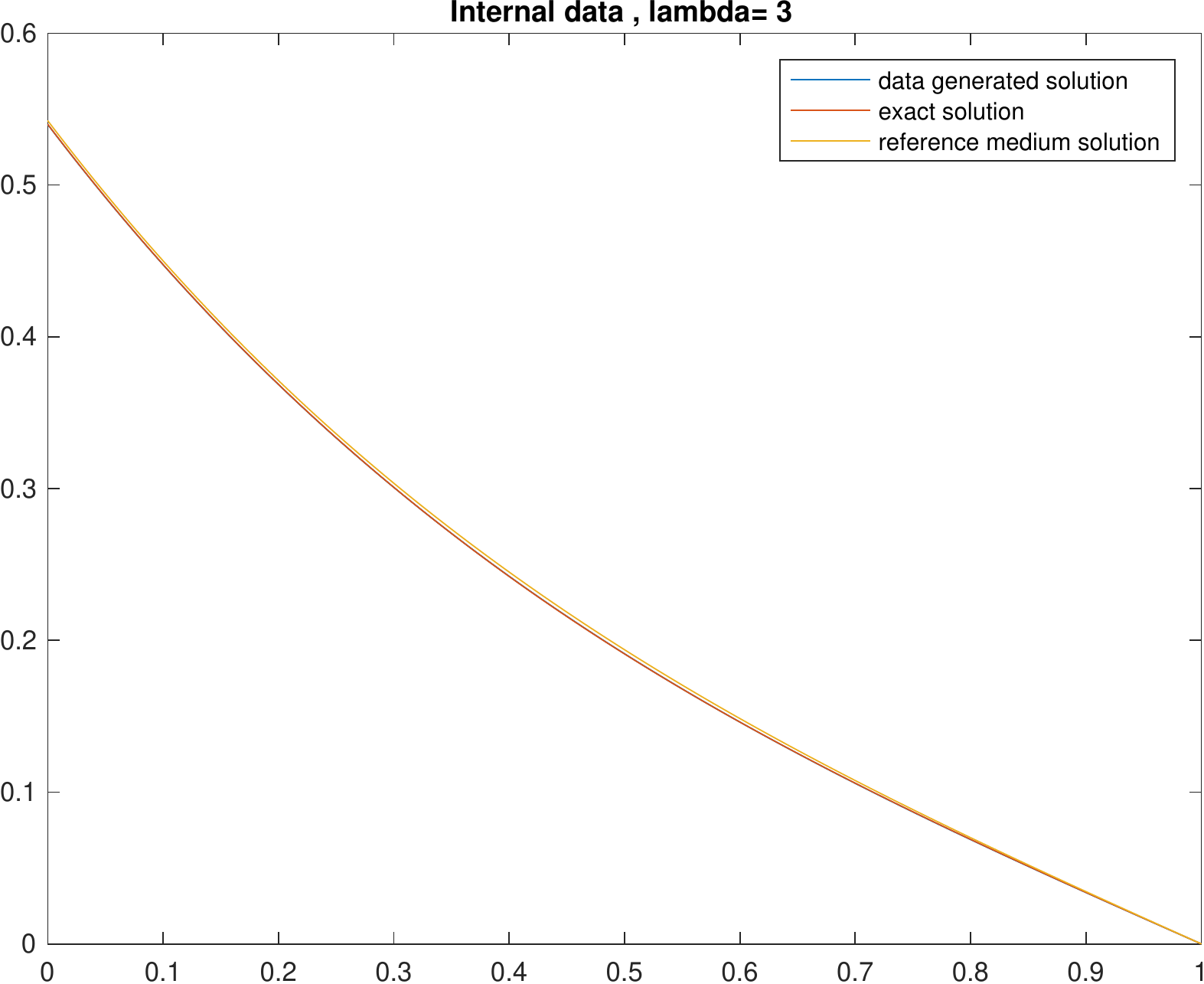}
\includegraphics[scale=.34]{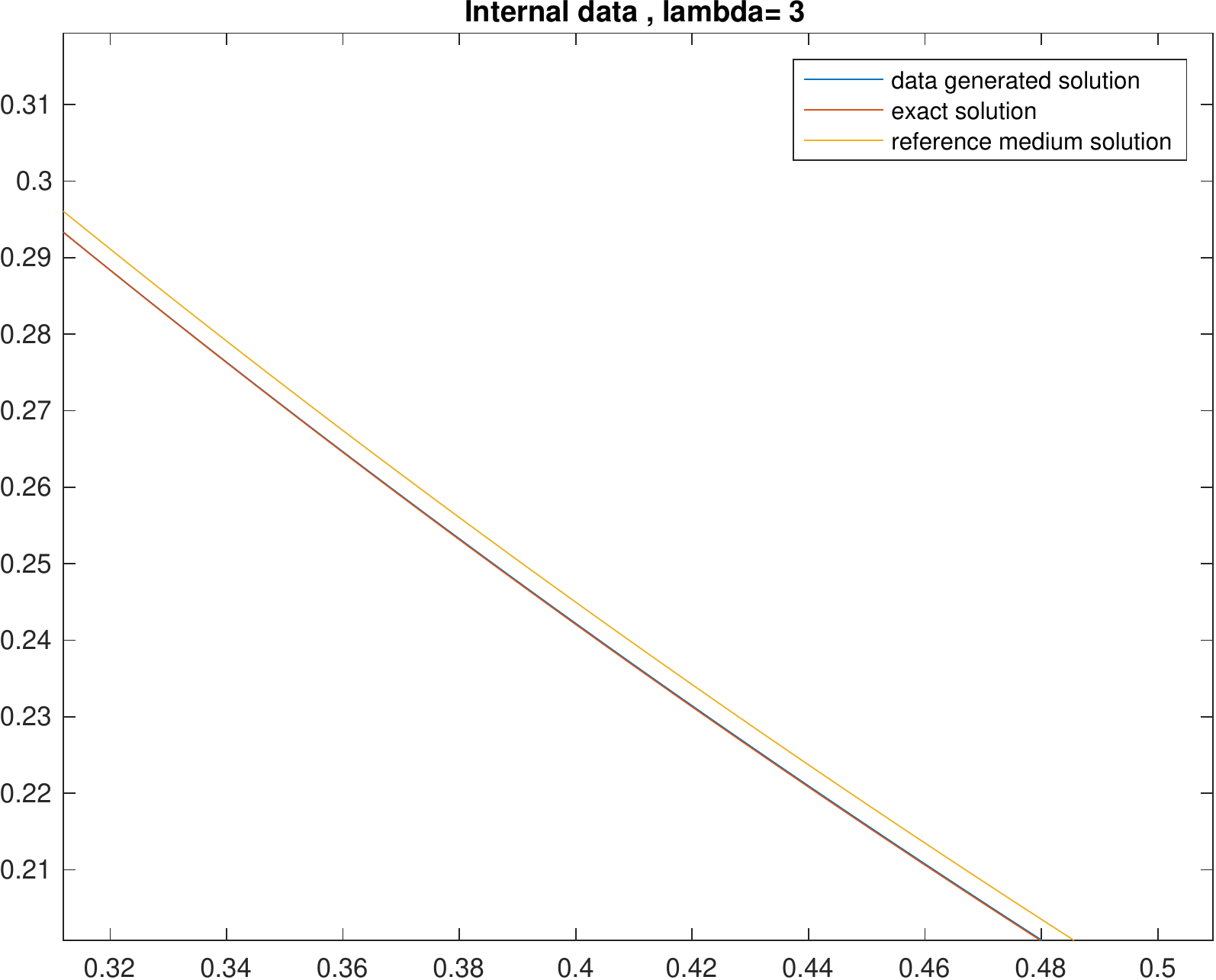}
\caption{Left: Internal solution for arbitrarily chosen spectral value $\lambda=3$ generated from data, compared to exact perturbed and reference medium solutions. Right: Zoom in of same. }
\label{internal1d}
\end{figure}

\begin{figure} 
\centering
\includegraphics[scale=.4]{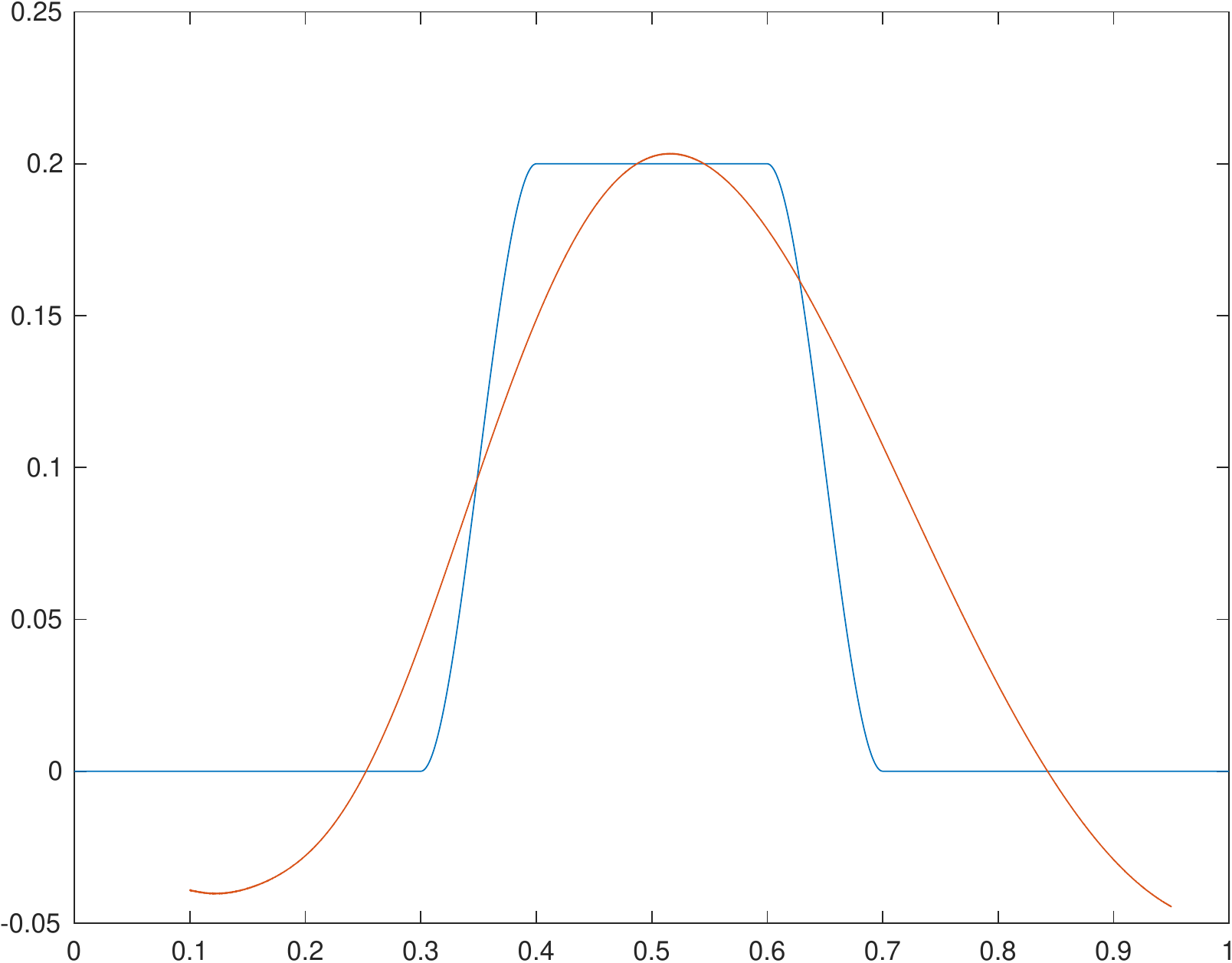}
\caption{Reconstruction (red)  of piecewise cubic $q(x)$  (blue) using derivatives of data generated internal solution. }
\label{recon1d1}\end{figure}

\begin{figure} 
\centering
\includegraphics[scale=.4]{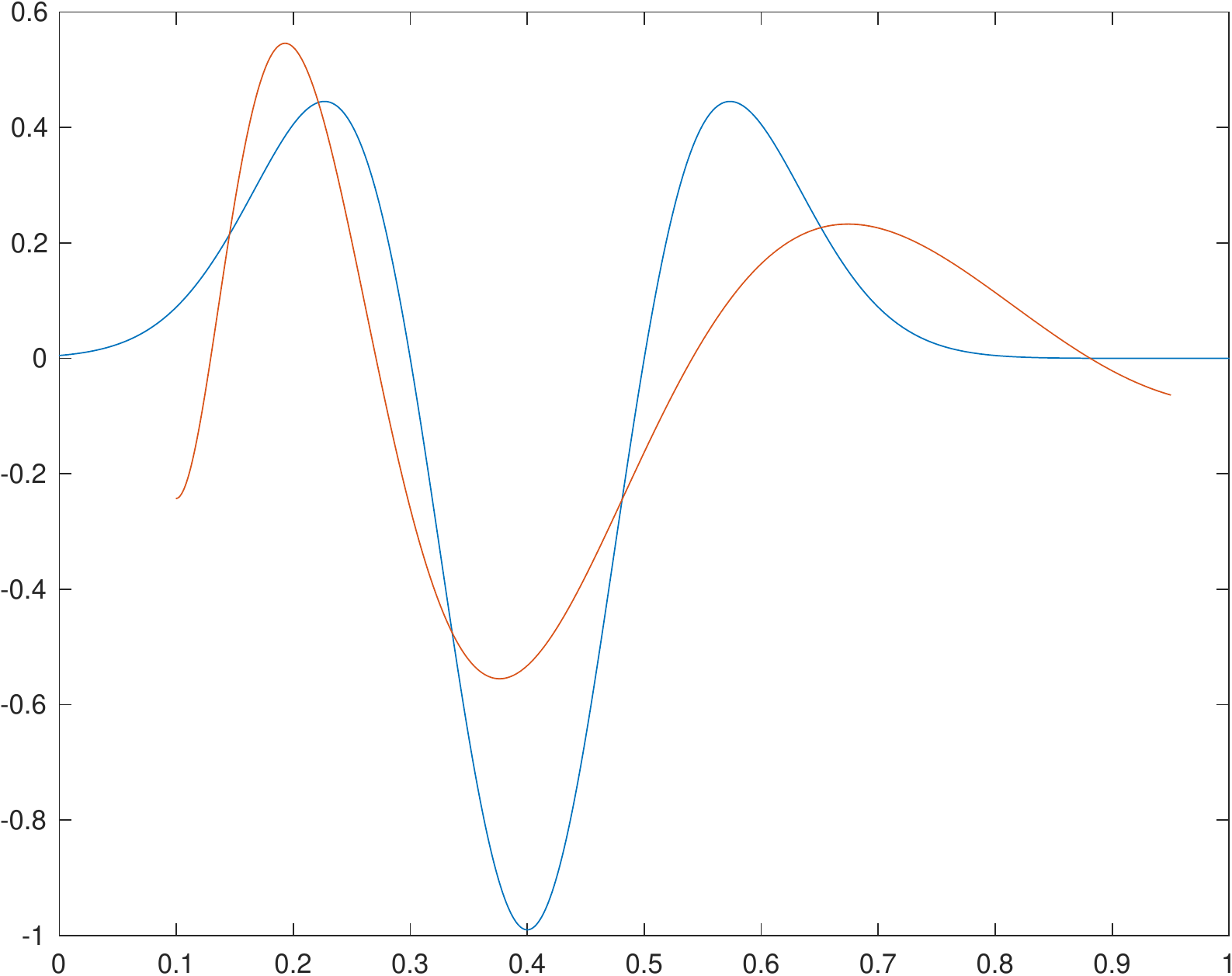}
\caption{Reconstruction (red) of smooth $q(x)$ (blue) using derivatives of data generated internal solution. }
\label{recon1d2}\end{figure}

In summary, the spectral data gives the Galerkin system ROM for (\ref{1d1side}) corresponding to spectral snapshots. Simultaneously, we can do the same for a reference medium. Lanczos orthogonalization localizes the spectral snapshots, forming a basis which depends very weakly on the medium. In Figure \ref{basis}, we show an example of the orthogonalized basis functions for a piecewise cubic $q$ varying from $0$ up to $.2$ and back down to $0$. Here they are plotted against the orthogonalized basis functions for the reference medium, and appear indistinguishable. 
Therefore, although we do not have the internal spectral snapshots for the true medium, in the localized basis they are very close to those of the reference medium. For any given $\lambda$ now, by solving the Galerkin system, we can get the coefficients, and this yields a good approximation to internal data.  See Figure \ref{internal1d} for an example of an internal solution.  One possibility is to use the internal solution to do inversion- for example by calculating $(\tilde{u}^{\prime\prime}-\lambda \tilde{u})/\tilde{u}\approx q $.  Note that for positive $\lambda$, $\tilde{u}$ is never zero.  We try this simple approach in Figure \ref{recon1d1} and Figure \ref{recon1d2}.

\section{Higher dimensional problems.}
\subsection{Galerkin model from spectral data in higher dimensions.}
Again we follow the Loewner framework. Consider the full problem (\ref{2dprob1}) with given data (\ref{data}). 
Assuming that each Neumann source function $g_r\in H^{-1/2}(\partial \Omega)$ and $q\in L^\infty(\Omega)$, the variational formulation of (\ref{2dprob1}) is: Find $u_i^r\in H^1(\Omega)$ such that 
$$\int_\Omega \nabla u_i^r \cdot \nabla \phi +\int_\Omega q u_i^r \phi +b_i\int_\Omega u_i^r \phi = \int_{\partial\Omega} g_r \phi $$
for all $\phi\in H^1(\Omega)$.  Generating a Galerkin subspace with these $m*K$ exact solutions  $$ V_{mK}=\mbox{span} \{  \{ u_i^r \}_{i=1,\ldots,m, r=1,\ldots,K}  \} $$ and using them as test functions, we have that
\begin{eqnarray}\int_\Omega \nabla u_i^r \cdot \nabla u_j^l  +\int_\Omega q u_i^r u_j^l +b_i\int_\Omega u_i^r u_j^l &=& \int_{\partial\Omega} g_r u_j^l \\
&=& F^j_{lr} \end{eqnarray}
for all $i,j=1,\dots, m$ and $r,l=1,\ldots,K$. That is,
$$ S_{irjl} + b_i M_{irjl}= F^j_{lr} $$
where $S$ and $M$ are the stiffness and mass matrices respectively.
From this and the same derivation as in one dimension, one obtains them directly from the data
\begin{equation} \label{M2deq}  M_{irjl}= { F^j_{lr}- F^i_{lr} \over{b_i-b_j}}, \end{equation} 
\begin{equation} \label{M2deq2} M_{iril}= - DF^i_{lr}, \end{equation}
 \begin{equation} \label{S2deq}  S_{irjl}= {b_j F^j_{lr}- b_i F^i_{lr} \over{b_j-b_i}}, \end{equation} 
 and
\begin{equation} \label{S2deq2} S_{iril}= (\lambda F_{rl})^\prime (b_i) .\end{equation}
\subsection{Block Lanczos orthogonalization and generation of the block tridiagonal reduced order model.}
Note the above $S$ and $M$ can be viewed as $m\times m$ matrices with entries that are $K\times K$ blocks. First we consider our data functional $\delta_F (\phi): H^1(\Omega)\rightarrow \mathbb{R}^K $ which is defined by
$$ \left[ \delta_F(\phi) \right]_l := \int_{\partial\Omega} \phi g_l \ \ \mbox{ for } \ l= 1,\ldots, K .$$ 
Similarly to what was done in one dimension, we find the projection of each component of $\delta_F(\phi) $ onto $V_{mK}$.  That is, we look for $\delta_{mK}^l \in V_{mK}$ such that
$$\int_\Omega \delta_{mK}^l  u_j^r = \int_{\partial \Omega} u_j^r g_l = F^j_{lr} .$$
That is, we identify $\delta_{mK}^l$ with its coefficients $\vec{d^l}$ in the basis for $V_{mK}$;
$$ \delta_{mK}^l = \sum_{i=1,\dots, m; r=1,\ldots, K}  d^l_{ir} u_i^r , $$
which again can be obtained directly by solving the $m\times m$ matrix system of $K\times K $ blocks $$ M \vec{d} = F.$$ We now use this as our starting vector for block Lanczos orthogonalization with respect to the basis generated by powers of the matrix $A=M^{-1}S$:
$$B=\{  \vec{d}, A\vec{d}, A^2\vec{d}, \ldots, A^{m-1}\vec{d} \}.$$
and we orthogonalize using Gram-Schmidt with respect to the mass matrix $M$ inner product 
$$\langle \vec{x},\vec{y} \rangle_M := \langle M\vec{x},\vec{y}\rangle .$$ This will yield again continuous $L^2$ orthogonality of the basis functions. Note that $A$ is symmetric with respect to this inner product and will be {\it block tridiagonal } in the new basis for $V_{mK}$:
\begin{equation}\label{orthogbasis2d} V_{mK} =\mbox{span}\{  \{ \hat{u}_i^r(x) \}_{i=1,\ldots,m; r=1,\ldots,K}  \}. \end{equation}
In this new basis, $\hat{M}$ will be identity and $\hat{S}$ will be block tridiagonal, yielding a sparse, spectrally converging ROM generated entirely from the data.

\subsection{Embedding into the continuous problem and generation of internal data.}
We now describe the algorithm to generate the internal solutions, which is simply a generalization of that from Section 2.3.  Here we assume we know or can compute solutions to the reference problem 
\begin{eqnarray} \label{2dKsource0} -\Delta (u_i^r)^0 +q_0 (u_i^r)^0 +b_i (u_i^r)^0 & = & 0 \ \ \mbox{ in } \ \ \ \Omega \\   {\partial (u_i^r)^0 \over {\partial \nu}} & = & g_r \ \ \mbox{ on } \partial \Omega \nonumber 
\end{eqnarray}
~\\

\noindent
\underline{\bf Algorithm}
\begin{enumerate}
\item Read data (again here synthetically generated) $$F^i_{rl}= \int_{\partial\Omega} u_i^r g_l $$ and  $$ {dF^i_{rl}\over{d\lambda}}$$ corresponding to some positive $\lambda=b_i$ for $i=1,\ldots,m$  and $K$ sources/receivers  $\{ g_r\} $ for the perturbed problem (\ref{2dprob1}). 

\item Compute all of the corresponding solutions $(u_i^r)^0$ for $i=1,\ldots,m, r=1,\ldots, K$ for the reference problem (\ref{2dKsource0}) and the corresponding reference data sets $F^0$ and $DF^0$. 

\item Use (\ref{M2deq}, \ref{M2deq2}, \ref{S2deq}, \ref{S2deq2}) to generate the $m\times m$ systems of $K\times K$ block stiffness and mass matrices $S$ and $M$ for Galerkin system for the perturbed problem, and similarly generate $S_0$, $M_0$ for the reference problem. 

\item Compute projection of the vector of functionals $\delta_F(\phi)$  onto $$V_{mK}=\mbox{span} \{  \{ u_i^r \}_{i=1,\ldots,m, r=1,\ldots,K}  \} $$ and onto the corresponding reference Galerkin space $$ V_{mK}^0 = \mbox{span} \{  \{ (u_i^r)^0 \}_{i=1,\ldots,m, r=1,\ldots,K}  \},$$ that is, we compute the coefficient vectors $\vec{d}$ and $\vec{d^0}$ by solving the systems  $$M\vec{d} = F $$ and $$M_0\vec{d^0} = F^0 $$
respectively.  

\item Set $A=M^{-1} S$, and perform block Lanczos orthogonalization with respect to the $M$ inner product using
$$\{  \vec{d}, A\vec{d}, A^2\vec{d}, \ldots, A^{m-1}\vec{d} \}$$
and let  $\hat{S}$,  $\hat{M}=I$ be the Galerkin stiffness and mass matrices in this new basis $ \{ \hat{u}_i^r\} $ for $V_{mK}$. Note that $\hat{S}$ is {\it block} tridiagonal with $K\times K$ blocks.

\item Similarly, set $A_0=M_0^{-1} S_0$ and perform Lanczos orthogonalization with respect to the $M_0$ inner product on the basis 
$$\{  \vec{d}_0, A_0\vec{d}_0, A_0^2\vec{d}_0, \ldots, A_0^{m-1}\vec{d}_0 \}$$ 
to obtain the corresponding orthogonalized basis $\{ (\hat{u}_i^r)^0\} $ for $V_{mK}^0$.

\item Now choose any spectral value $\lambda$ and solve the Galerkin system in the new basis $ c= (\hat{S}+\lambda I)^{-1} \hat{F}   $ , so that for each $r$,  $u^r_G=\sum_i c^r_i \hat{u}^r_i $ is the Galerkin projection of the solution of (\ref{2dprob1}) with $b_i=\lambda$ onto $V_{mK}$. Since we don't know $\hat{u}^r_i $ from the data, approximate $ u^r_G $ by $$ \tilde{u}^r = \sum_i c^r_i (\hat{u}^r_i)^0.$$

\item Compute  $(\Delta\tilde{u}^r -\lambda \tilde{u}^r)/\tilde{u}^r \approx q $; or reserve $\{ \tilde{u}^r \} $ for another use. 

\end{enumerate}
\begin{remark} We note that just as the internal solutions can be computed for any spectral value $\lambda$, they also be computed for any Neumann data where one can read the above solutions $u_j^r$ on their support.  For any given source $g$, the data $$F^j_r = \int_{\partial\Omega} u_j^r g$$ is the right hand side for the Galerkin system in the original basis, which can be transformed to $\hat{F}$ by the same change of basis found in the above Lanczos procedure. 
\end{remark}

\begin{figure}
\centering
\includegraphics[scale=.34]{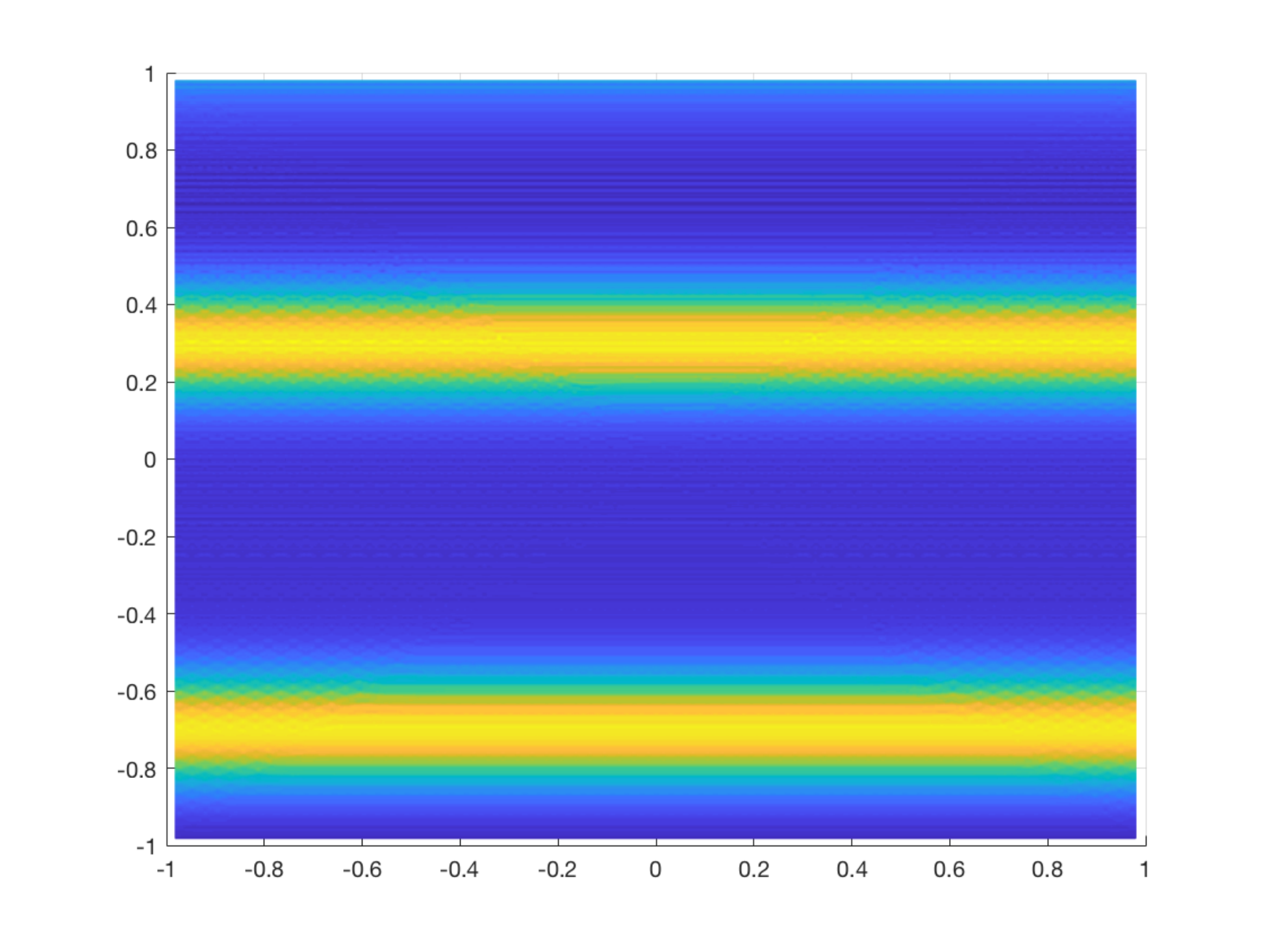}
\includegraphics[scale=.34]{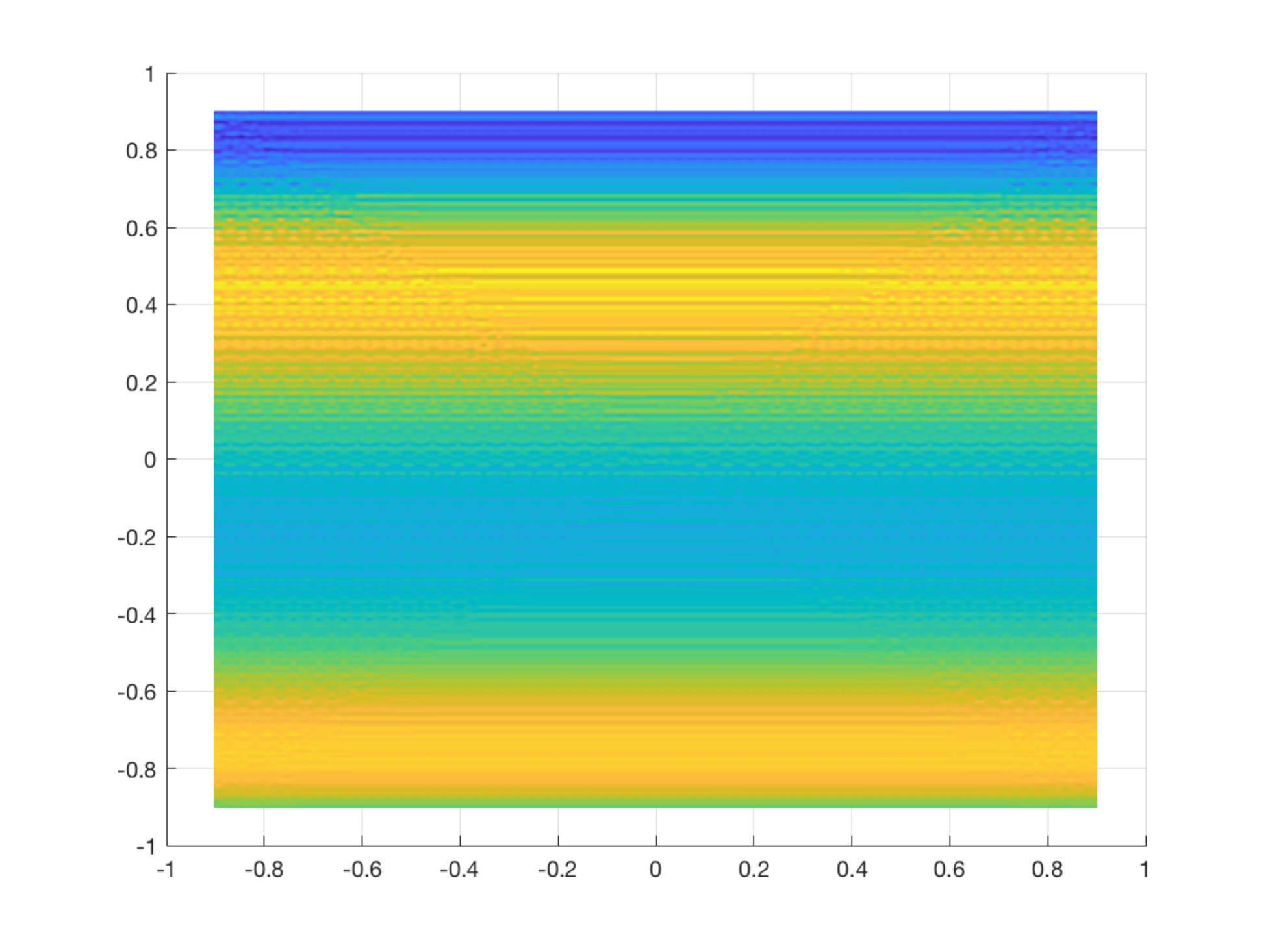}
\caption{Left: Layered medium with one source on each side of layering (two sided problem). Left: Original medium.  Right: Reconstruction using data generated internal solutions. }
\label{layered}
\end{figure}

\begin{figure}
\centering
\includegraphics[scale=.34]{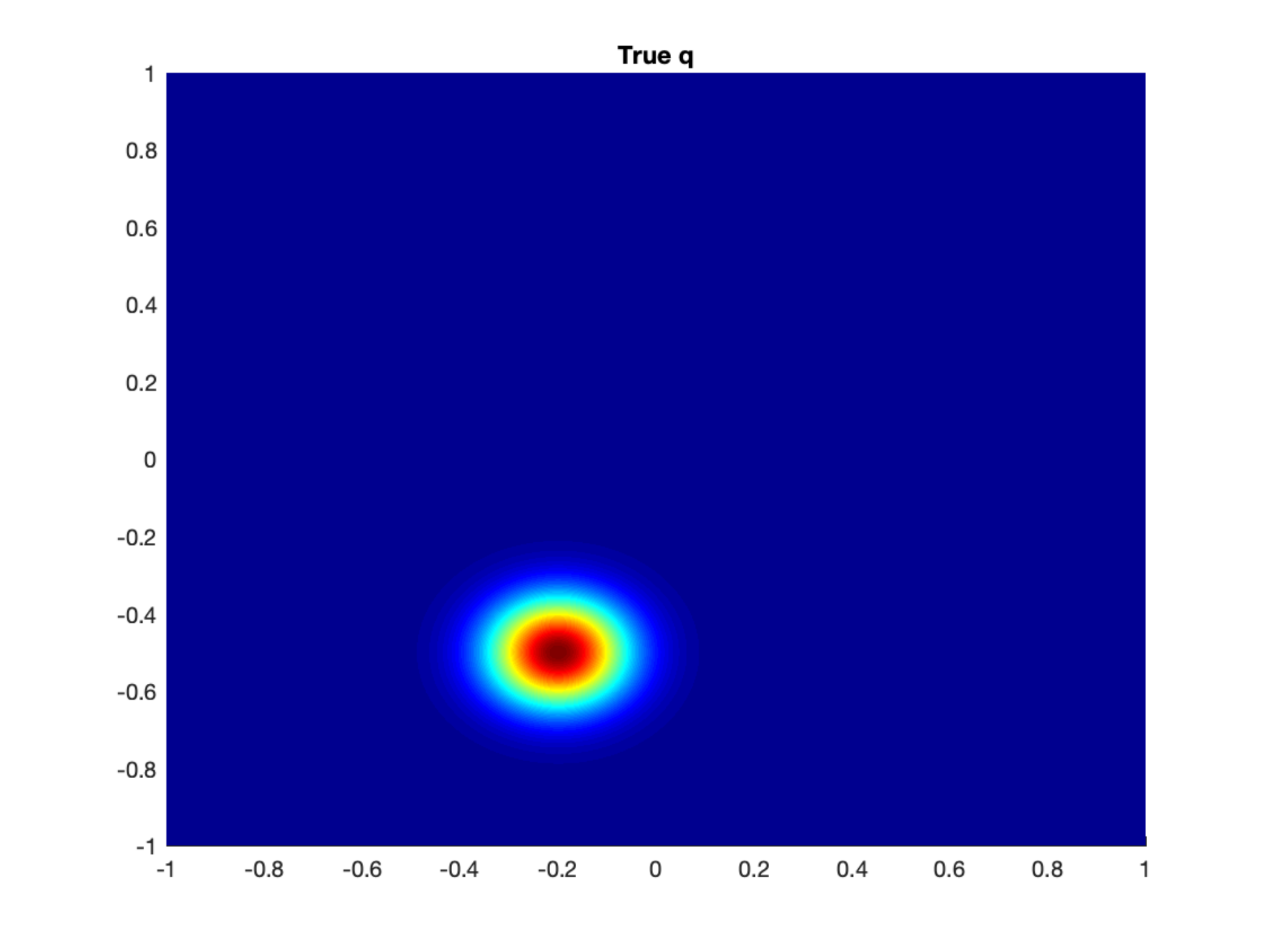}
\includegraphics[scale=.34]{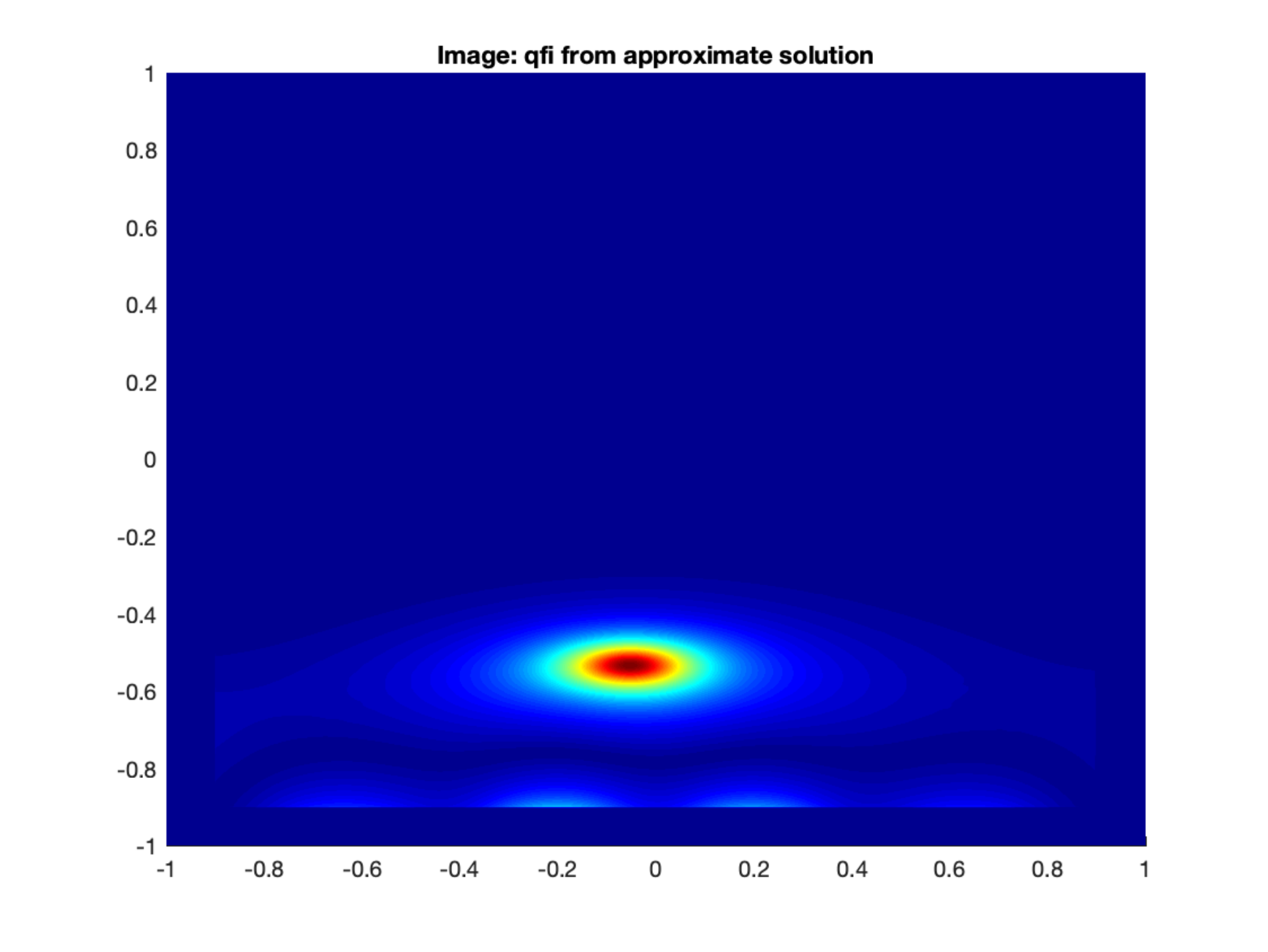}
\caption{Left: Smooth bump profile. Right: Reconstruction using four sources/receivers on bottom only, using data generated internal solution. }
\label{onebump}
\end{figure}

\begin{figure}
\centering
\includegraphics[scale=.34]{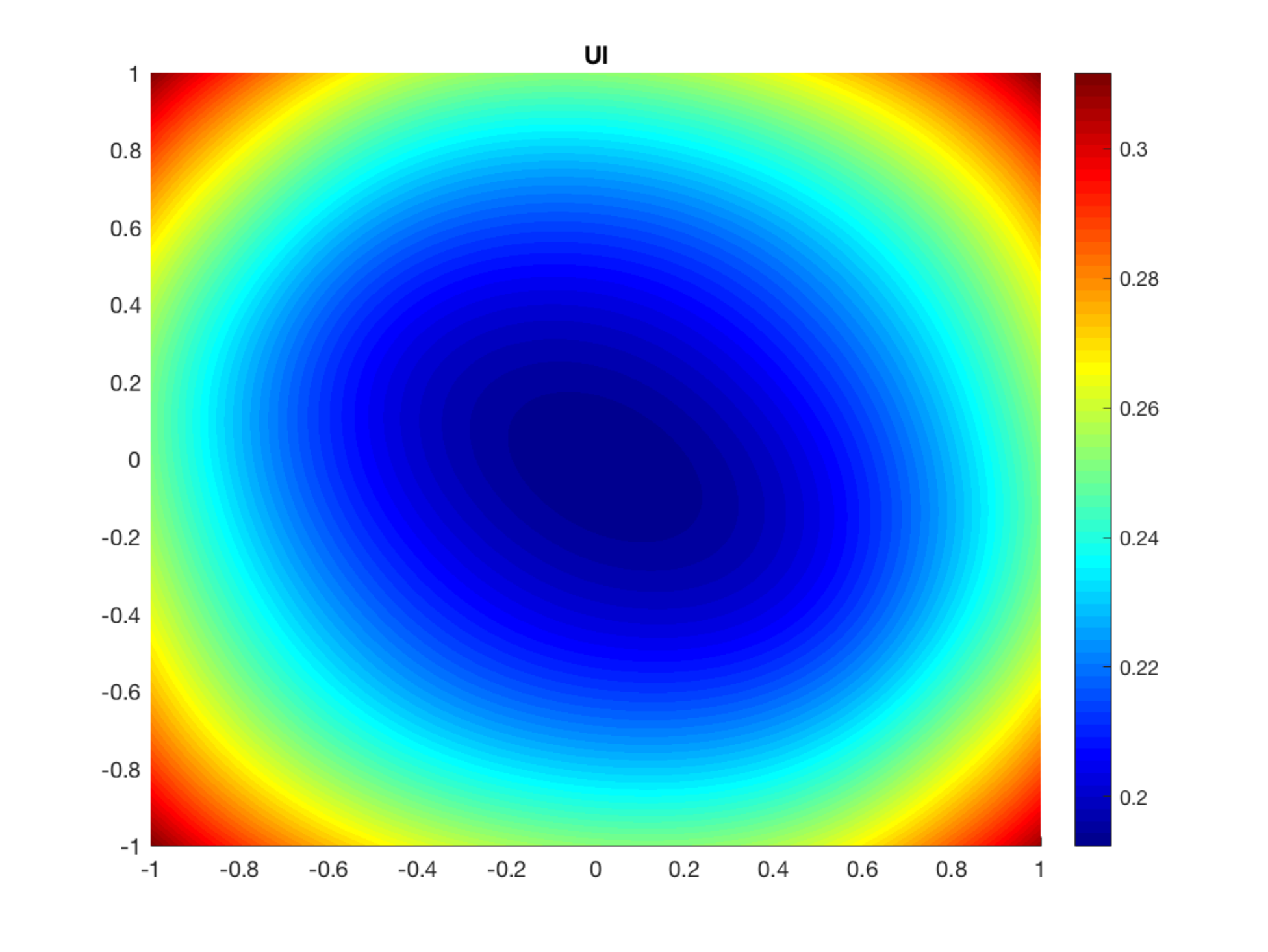}
\includegraphics[scale=.34]{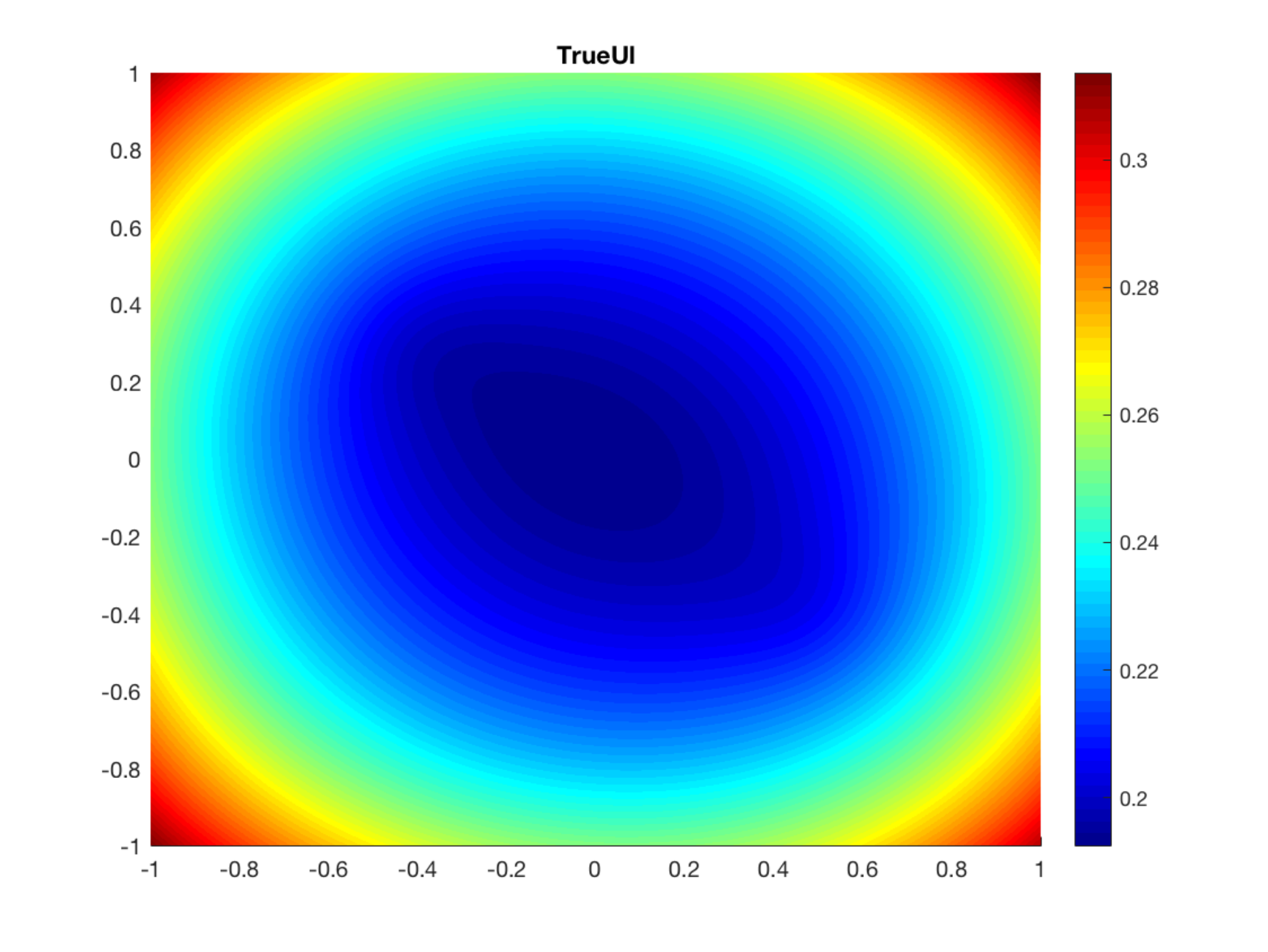}
\includegraphics[scale=.34]{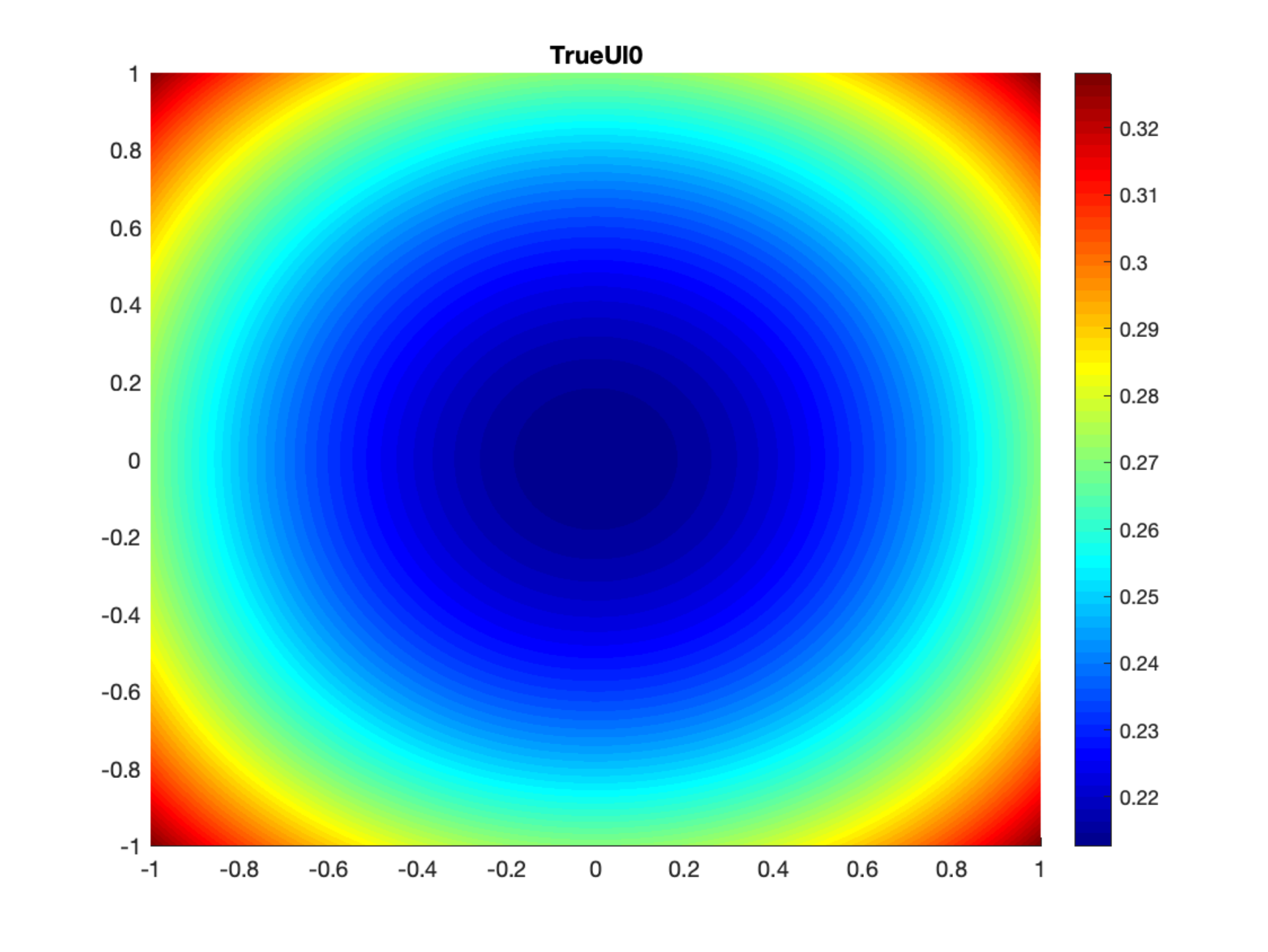}
\caption{Left: Data generated internal solution for two bumps. Right: Actual internal solution. Bottom: Background solution for comparison. The relative error between the true and data generated internal solutions is 0.003930. For comparison, the relative error between the true and reference medium internal solutions is 0.084794.}
\label{twobumpinternal}
\end{figure}

\begin{figure}
\centering
\includegraphics[scale=.34]{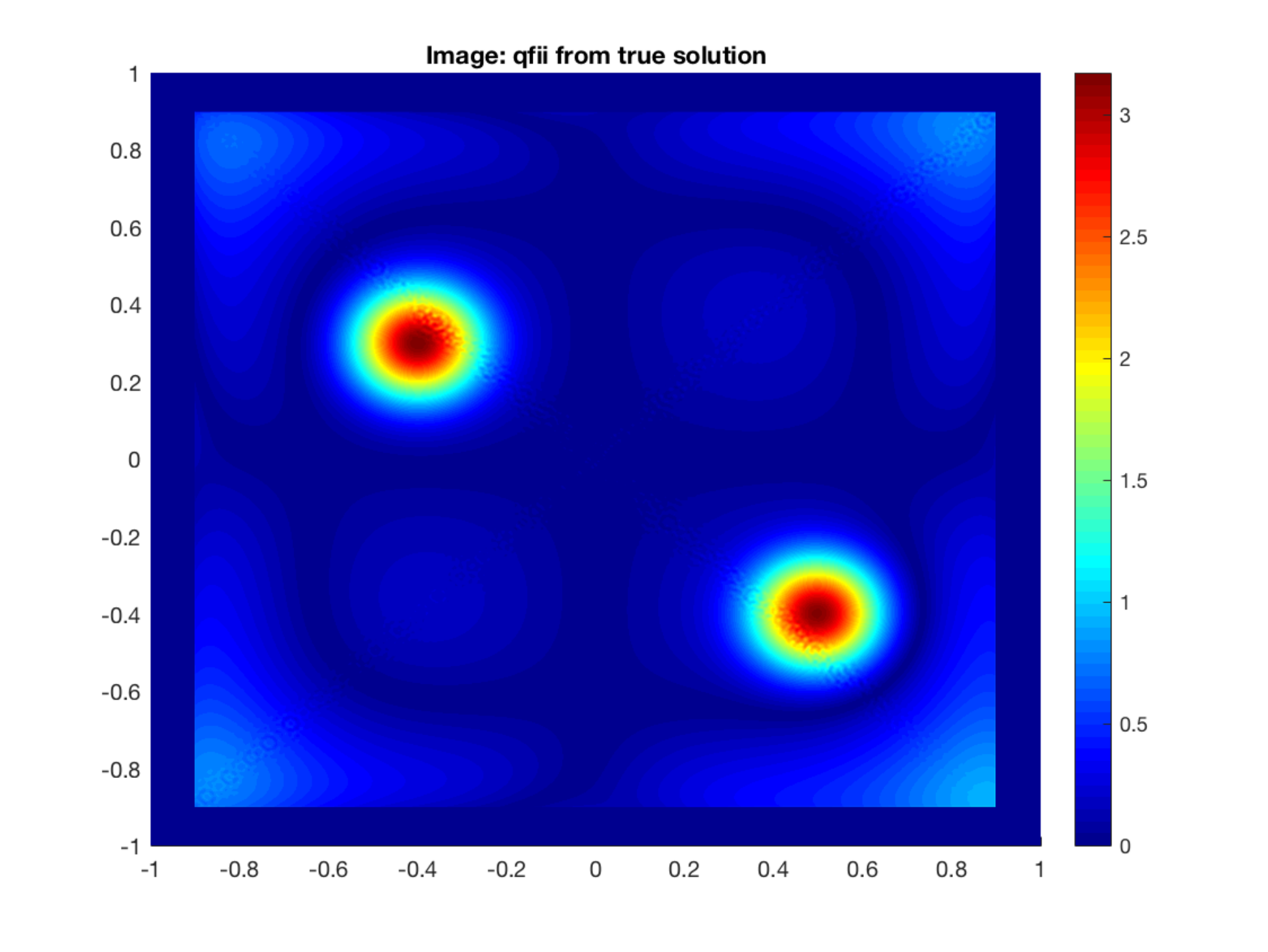}
\includegraphics[scale=.34]{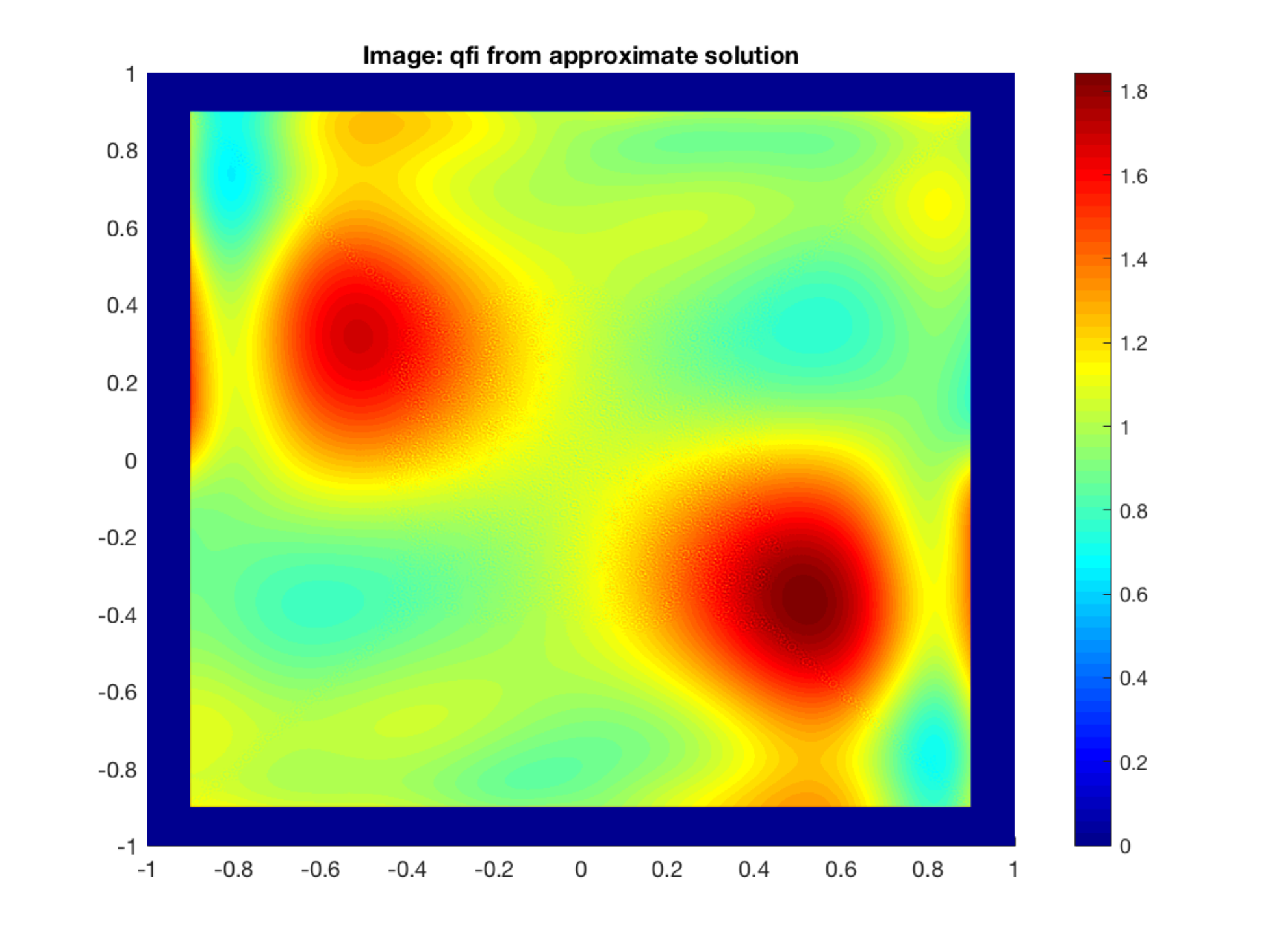}
\caption{Reconstruction of two bumps  . Eight sources total; two on each side, and six spectral values. }
\label{twobumpreconstruct}
\end{figure}

We now describe a few preliminary numerical experiments where we compute internal solutions from the data using the above described algorithm. In the first experiment, we consider a $q$ which is 1-d smoothly varying medium with two Gaussian bumps. For this case, we will use two sources; each a constant on one of the two sides of the direction of variation. We do not assume apriori that we know the medium is only one dimensionally varying. In this case, $K=2$ corresponding to the two sources, and we use six spectral values, yielding a ROM of total size $12\times 12$. We implement the above algorithm as described, and show the original medium and its reconstruction in Figure \ref{layered}.

In our next example, we consider a medium with a smooth two dimensional bump, and probe with four ($K=4$)  piecewise constant sources, all on one side of the domain. In this case we use $s=4$ spectral values, yielding a ROM of total size $16\times 16$. The results of the reconstruction are in Figure \ref{onebump}. 

Finally, we consider a two-bump medium and probe with $K=8$  piecewise constant sources; two on each side. We use $s=6$ spectral values. For this example we first examine the internal solutions in Figure \ref{twobumpinternal}. One sees that the data generated internal solutions capture the asymetries in the true solution that the background solution does not. The error in the data generated internal solution is just under $0.4\%$. In Figure \ref{twobumpreconstruct} we apply the Laplacian to obtain a reconstruction in which we see the two bumps. 

\section{Conclusions}

Reduced order Galerkin models for spectral domain problems can be generated directly from boundary data, and we saw in one dimension that these lead to the same boundary response as the corresponding spectrally matched grid. After Lanczos orthogonalization, the Galerkin model becomes tridiagonal and becomes everywhere the spectrally matched finite difference system. Most importantly, the orthogonalized basis functions depend only very weakly on the medium. This allows one to generate accurate internal solutions entirely from boundary data.  The approach extends nicely to higher dimensions, and the internal solutions remain highly accurate, and  can be differentiated to yield a fast direct inversion method. Furthermore, the internal solutions themselves have potential to apply some problems directly, for example in applications to internal field monitoring for medical ablation.   A rigorous analysis of the dependence of the orthogonalized basis functions on the medium is still needed and is the subject of future work.

\section*{Acknowledgments}

{ Borcea’s research is supported in part by the NSF grant DMS1510429 and in part by the U.S. Office of Naval Research under award number N00014-17-1-2057. Mamonov's research is supported in part by the NSF grant DMS-1619821 and in part by the U.S. Office of Naval Research under award number N00014-17-1-2057. Moskow's research is supported in part by NSF grant DMS1715425.}

\bibliography{galerkincitations} 
\bibliographystyle{plain} 

\end{document}